\documentclass[11pt]{article}
\usepackage{graphicx} 
\usepackage{geometry}
\usepackage{amssymb}
\usepackage{mathtools}
\usepackage{amsmath}
\usepackage{appendix}
\usepackage{amsthm}
\usepackage{hyperref}
\hypersetup{colorlinks,allcolors=black}
\usepackage{color}
\usepackage{bm}
\usepackage{dsfont}
\usepackage[backend = biber, style = numeric, url = false, doi = false]{biblatex}[=3.2]
\usepackage{comment}
\usepackage{float}
\DeclareMathOperator{\R}{\mathbb{R}}
\DeclareMathOperator{\E}{\mathbb{E}}
\usepackage{accents}
\numberwithin{equation}{section}
\newtheorem{theorem}{Theorem}[section]
\newtheorem{proposition}[theorem]{Proposition}

\newtheorem{remark}[theorem]{Remark}
\newtheorem{lemma}[theorem]{Lemma}
\newtheorem{assumption}{Assumption}

\newcommand{\widesim}[2][1.5]{
  \mathrel{\overset{#2}{\scalebox{#1}[1]{$\sim$}}}
}
\allowdisplaybreaks

\title{Particle method for the McKean-Vlasov equation with common noise}
\author{%
Théophile Le Gall\thanks{\small CEREMADE, CNRS, UMR 7534, Universit{\'e} Paris-Dauphine,
PSL University, 75016 Paris, France, \texttt{legall@ceremade.dauphine.fr}.}%
  }

\begin{document}

\maketitle

\begin{abstract}

This paper studies the numerical simulation of the solution to the McKean-Vlasov equation with common noise. We begin by discretizing the solution in time using the Euler scheme, followed by spatial discretization through the particle method, inspired by the propagation of chaos property. Assuming Hölder continuity in time, as well as Lipschitz continuity in the state and measure arguments of the coefficient functions $b$, $\sigma$ and $\sigma^0$, we establish the convergence rate of the Euler scheme and the particle method. These results extend those in \cite{NumMethMKV} for the standard McKean-Vlasov equation without common noise. Finally, we present two simulation examples : a modified conditional Ornstein Uhlenbeck process with common noise and an interbank market model presented in \cite{ren2024risk}.

\end{abstract}


\section{Introduction}

We consider the $\R^d$-valued McKean-Vlasov stochastic differential equation (SDE) with common noise defined for $t\in[0,T]$ by
\begin{equation} \label{eq:MKV}
    dX_t=b(t,X_t, \mathcal{L}^1(X_t))dt + \sigma(t,X_t,\mathcal{L}^1(X_t))dW_t + \sigma^0(t,X_t, \mathcal{L}^1(X_t))dW_t^0,
\end{equation}
where, for some $T>0$, $W = (W_t)_{t\in[0,T]}$ and $W^0 = (W^0_t)_{t\in[0,T]}$ are two independent Brownian motions respectively called idiosyncratic noise and common noise. The coefficient $b$ is a mapping from $[0,T]\times\R^d\times\mathcal{P}_p(\R^d)$ to $\R^d$ where $\mathcal{P}_p(\R^d)$ is the space of probability measure having a $p$-th finite moment for $p\geq 1$. We endow $\mathcal{P}_p(\R^d)$ with the $p$-Wasserstein metric defined later in \eqref{eq:def_wasserstein_distance}. The coefficients $\sigma, \sigma^0$ are mappings from $[0,T]\times\R^d\times\mathcal{P}_p(\R^d)$ to $\R^d\otimes\R^q$ which represent respectively the intensity of the idiosyncratic noise $W$ and the common noise $W^0$. The notation $\mathcal{L}^1(\cdot)$ represents for the conditional law given the trajectory of the common noise (see further \eqref{eq:def-L1X}, Proposition \ref{prop:lemma 2.4 Carmona Delarue} and \ref{lem:L1_cond_law}). The initial condition of Equation \eqref{eq:MKV} is a random variable $X_0$ independent of $W$ and $W^0$.\medskip

The McKean-Vlasov equation, initially introduced by H. McKean \cite{McKean} is a nonlinear partial differential equation (PDE) associated with a class of stochastic differential equations (SDEs) where the drift and diffusion coefficients depend not only on the time and the state of the process, but also on its marginal laws (see e.g. \cite{TopicsPoC}). The distribution-dependent structure of the McKean-Vlasov equation is extensively applied for modeling phenomena in statistical physics (see e.g. \cite{Bossy_2022}, \cite{Martzel_2001}), mathematical biology (see e.g. \cite{baladron2012mean}), social sciences and quantitative finance, both frequently driven by advancements in mean field games and interacting diffusion models (see e.g. \cite{cardaliaguet2018mean}, \cite{LASRY2018886}). The McKean-Vlasov equation equipped with a common noise (see further \eqref{eq:MKV} for a precise definition) was first introduced in \cite{WellPosednessMFGCommonNoise}, \cite{MFGwithCommonNoise}, \cite{LinQuadControl_MFGCommonNoise} and \cite{TranslationInvariantMFGCommonN}, where the term \textit{common noise} served to model a type of shared risk in a particle system. Several papers such as \cite{MR4022284} or \cite{MR4046528} explore how the introduction of the common noise can restore uniqueness in mean-field games, which are derived from deterministic differential games involving a large number of players.

This paper aims to develop a numerical method for the McKean-Vlasov equation with common noise, accompanied by an analysis of the associated convergence rate. For the standard McKean-Vlasov equation without common noise, under Lipschitz continuity assumptions on the coefficient functions, we refer to \cite{bossy1997stochastic} and \cite{NumMethMKV} for the simulation of the solution to the SDE, and to \cite{antonelli2002rate} and \cite{DensitySimuMKV} for the estimation of the density solution to the McKean-Vlasov PDE. Additionally, recent advancements in handling super-linear growth coefficient functions can be found in  \cite{Reis_SimuMKV}, \cite{CHEN2022127180} among others. For the simulation of the invariant measure, we refer to \cite{chassagneux2024computing}. For the McKean-Vlasov equation with common noise, a recent study \cite{MilsteinSchemeMKV} provides the convergence rate of a numerical scheme in a different setting from the one considered in this paper. For a detailed comparison, see further Remark \ref{rem:comparaison}.

\noindent

Following the construction as presented in \cite{bossy1997stochastic} and \cite{NumMethMKV} for standard McKean-Vlasov equation without common noise, our approach employs the Euler scheme, defined further in \eqref{eq:euler scheme MKV}, as a temporal discretization, and the particle method as a spatial discretization, defined further in \eqref{eq:euler particle system}, that was introduced in \cite{NumMethMKV} and we extend it to account for the case with common noise. Notice that the addition of common noise needs further careful consideration, particularly in terms of the conditional distributions given the common noise in the measure argument of the coefficient functions. In this context, the empirical measure serves as an estimator for the law of the solution process, conditioned on the common noise. This method relies on the propagation of chaos property, initially introduced by Kac \cite{kac1956foundations} and further studied in \cite{lacker2018strong}, \cite{lacker2023hierarchies}, and \cite{TopicsPoC}.

\subsection{Probabilistic settings}

In this paper, we consider two filtered probability spaces $\big(\Omega^0, \mathcal{F}^0,\mathbb{F}^0, \mathbb{P}^0\big)$ and $(\Omega^1, \mathcal{F}^1,\mathbb{F}^1, \mathbb{P}^1)$ satisfying the usual conditions, where $\mathbb{F}^0 = (\mathcal{F}^0_t)_{t\in[0,T]}$ and $\mathbb{F}^1 = (\mathcal{F}^1_t)_{t\in[0,T]}$. In addition, we provide $W^0=(W^0_t)_{t\in[0,T]}$ and $W=(W_t)_{t\in[0,T]}$ two $q$-dimensional $\mathbb{F}^0$-adapted and $\mathbb{F}^1$-adapted Wiener processes respectively supported on $\big(\Omega^0, \mathcal{F}^0, \mathbb{P}^0\big)$ and $(\Omega^1, \mathcal{F}^1, \mathbb{P}^1)$ which respectively represent the common noise and the idiosyncratic noise. We introduce naturally the product space $(\Omega,\mathcal{F},\mathbb{F}, \mathbb{P})$ where $\Omega = \Omega^0 \times \Omega^1$, $(\mathcal{F},\mathbb{P})$ is the completion of $(\mathcal{F}^0\otimes\mathcal{F}^1, \mathbb{P}^0\otimes\mathbb{P}^1)$ and $\mathbb{F}=(\mathcal{F}_t)_{t\in[0,T]}$ is the complete and right-continuous augmentation of $(\mathcal{F}^0_t\otimes\mathcal{F}^1_t)_{t\in[0,T]}$. Consider a random variable $X$ defined on the filtered probability space $(\Omega, \mathcal{F}, \mathbb{F}, \mathbb{P})$. Then, for $\mathbb{P}^0$-a.e., $\omega^0 \in \Omega^0$, $X(\omega^0,\cdot)$ is a random variable on $(\Omega^1, \mathcal{F}^1, \mathbb{P}^1)$ (see e.g. \cite[Section 2.1.3]{Carmona_Delarue2}). In particular, we may define 
\begin{equation}\label{eq:def-L1X}
\mathcal{L}^1(X) : \omega^0\in \Omega^0\mapsto \mathcal{L}(X(\omega^0,\cdot))\in\mathcal{P}(\R^d),
\end{equation}
for almost every $\omega^0\in \Omega^0$. On the exceptional event where $\mathcal{L}(X(\omega^0,\cdot))$ cannot be computed, we may assign it arbitrary values in $\mathcal{P}(\R^d)$.

\begin{proposition}[{\cite[Lemma 2.4]{Carmona_Delarue2}}] \label{prop:lemma 2.4 Carmona Delarue}
Given a random variable $X: (\Omega, \mathcal{F}, \mathbb{P})\rightarrow (\R^d, \mathcal{B}(\R^d))$, the mapping $\mathcal{L}^1(X)$ defined by \eqref{eq:def-L1X} is almost surely well defined under $\mathbb{P}^0$, and forms a random variable from $(\Omega^0, \mathcal{F}^0, \mathbb{P}^0)$ into $\mathcal{P}(\R^d)$ endowed with its Borel $\sigma$-field generated by the Lévy-Prokhorov metric (see \cite[Section 5.1.1 and Proposition 5.7]{Carmona_Delarue1}). Moreover, the random variable  $\mathcal{L}^1(X)$ provides a conditional law of $X$ given $\mathcal{F}^0$.
\end{proposition}

\noindent
We endow $\mathcal{P}_p(\R^d)$ with the $p$-Wasserstein metric 
\begin{equation} \label{eq:def_wasserstein_distance}
    \mathcal{W}_p(\mu,\nu) = \underset{\pi\in\Pi(\mu,\nu)}{\inf}\bigg[\int_{\R^d\times \R^d}|x-y|^p\pi(dx,dy) \bigg]^{\frac{1}{p}},
\end{equation} 
where $\Pi(\mu,\nu)$ denotes the set of probability measure on $\R^d\times \R^d$ with respective marginals $\mu$ and $\nu$.

\subsection{Construction of the Euler scheme and particle method} \label{subsection : def scheme}

Let $M$ be the number of time discretization and $h=\frac{T}{M}$ be the time step. For every $m\in\{0,\ldots,M\}$, we define $t_m = hm$. We consider $W^1, \ldots, W^N$  i.i.d. copies of the Brownian motion $W$, and define the re-normalized increments  $Z_m^i,\; Z_m,\; Z_m^0,\; 1\leq m\leq M,\; 1\leq i \leq N$ as follows 
\[Z^i_{m+1} = \frac{1}{\sqrt{h}} \big(W^i_{t_{m+1}} - W^i_{t_m}\big), \:Z^0_{m+1} = \frac{1}{\sqrt{h}} \big(W^0_{t_{m+1}} - W^0_{t_m}\big), \:Z_{m+1} = \frac{1}{\sqrt{h}} \big(W_{t_{m+1}} - W_{t_m}\big).\]
The theoretical Euler scheme of the McKean-Vlasov equation with common noise \eqref{eq:MKV} is defined, for $0\leq m\leq M-1$, by $\bar{X}_0=X_0$ and
\begin{equation} \label{eq:euler scheme MKV}
    \begin{aligned}
        \bar{X}^M_{t_{m+1}}=\bar{X}_{t_m}^M &+ h\cdot b\big(t_m,\bar{X}^M_{t_m}, \mathcal{L}^1\big(\bar{X}_{t_m}^M\big)\big) + \sqrt{h}\cdot \sigma\big(t_m,\bar{X}^M_{t_m}, \mathcal{L}^1\big(\bar{X}_{t_m}^M\big)\big)Z_{m+1}\\ &+ \sqrt{h}\cdot \sigma^0\big(t_m,\bar{X}^M_{t_m}, \mathcal{L}^1\big(\bar{X}_{t_m}^M\big)\big)Z^0_{m+1},
    \end{aligned}
\end{equation}
equipped with its natural continuous extension defined, for every $t\in[t_m, t_{m+1})$, by 
\begin{equation} \label{eq:continuous euler scheme MKV}
    \begin{aligned}
        \bar{X}^M_t=\bar{X}_{t_m}^M &+ b\big(t_m,\bar{X}^M_{t_m}, \mathcal{L}^1\big(\bar{X}_{t_m}^M\big)\big)(t-t_m) + \sigma\big(t_m,\bar{X}^M_{t_m}, \mathcal{L}^1\big(\bar{X}_{t_m}^M\big)\big)(W_t-W_{t_m})\\ &+ \sigma^0\big(t_m,\bar{X}^M_{t_m}, \mathcal{L}^1\big(\bar{X}_{t_m}^M\big)\big)(W^0_t-W^0_{t_m}).
    \end{aligned}
\end{equation}

\noindent
At each time step $t_m$, we build an $N$-particle system $(\bar{X}_{t_m}^{1,N}, \ldots,\,\bar{X}_{t_m}^{N,N})$ such that for every $i \in \{1,\ldots,N\}$, we have
\begin{equation} \label{eq:euler particle system}
    \left\{
    \begin{array}{ccl}
        \bar{X}_{t_{m+1}}^{i,N} &=& \bar{X}_{t_{m}}^{i,N} + h\cdot b\big({t_m},\bar{X}_{t_m}^{i,N}, \bar{\mu}_{t_m}^N\big) + \sqrt{h}\cdot \sigma\big({t_m},\bar{X}_{t_m}^{i,N}, \bar{\mu}_{t_m}^N\big)Z^i_{m+1} \smallskip \\
        && \quad + \sqrt{h}\cdot \sigma^0\big({t_m},\bar{X}_{t_m}^{i,N}, \bar{\mu}_{t_m}^N\big) Z^{0}_{m+1} \medskip \\
        \bar{\mu}_{t_m}^N &=& \frac{1}{N}\sum_{i=1}^N \delta_{\bar{X}_{t_m}^{i,N}}.
    \end{array}\right.
\end{equation}
also equipped with its natural continuous extension defined, for every $t\in[t_m, t_m+1)$, by 
\begin{equation}\label{eq:particle method continuous scheme}
    \bar{X}^{i,N}_{t} = \bar{X}^{i,N}_{t_m} + \int_{t_m}^tb(t_m, \bar{X}^{i,N}_{t_m}, \bar{\mu}^N_{t_m})ds + \int_{t_m}^t \sigma(t_m, \bar{X}^{i,N}_{t_m}, \bar{\mu}^N_{t_m})dW^i_s + \int_{t_m}^t \sigma^0(t_m, \bar{X}^{i,N}_{t_m}, \bar{\mu}^N_{t_m})dW^0_s.
\end{equation}
\noindent
In the system \eqref{eq:euler particle system}, at each time step $t_m$, the particles $\bar{X}_{t_{m}}^{i,N}$, $1\leq i\leq N$, have interaction through the empirical measure $\bar{\mu}_{t_m}^N$. The idea of the particle method is to use $\bar{\mu}_{t_m}^N$ as an estimator of $\mathcal{L}^1\big(\bar{X}_{t_m}^M\big)$ in definition \eqref{eq:euler scheme MKV}, at each time step $t_m$, $0\leq m\leq M$.

\subsection{Assumptions and main results}

The main results of this paper will be established under the following assumptions, which are assumed to be held for a fixed $p\in [2,\infty)$.

\begin{assumption}
\label{assumption 1}
The random variable $X_0$ is defined on $(\Omega^1,\mathcal{F}^1, \mathbb{P}^1)$ such that $\E^1(|X_0|^p) < \infty.$
\end{assumption}

\begin{assumption}
\label{assumption Lipschitz} 
    The coefficient mappings $b$, $\sigma$ and $\sigma^0$ are continuous in time and Lipschitz continuous in the state and measure arguments, that is, there exists a constant $L>0$ such that for every $t\in [0,T],\, x,y\in\R^d$ and $\mu,\nu\in\mathcal{P}_p(\R^d)$, we have
    \begin{equation*}
        \max(|b(t,x,\mu) - b(t,y,\nu)|, |\sigma(t,x,\mu)-\sigma(t,y,\nu)|, |\sigma^0(t,x,\mu)-\sigma^0(t,y,\nu)|)
        \leq L\big(|x-y| + \mathcal{W}_p(\mu,\nu)\big).
    \end{equation*}
\end{assumption}

\noindent
Assumptions \ref{assumption 1} and \ref{assumption Lipschitz} guarantee the existence and strong uniqueness of a solution $X=(X_t)_{t\in[0,T]}$ to the McKean-Vlasov equation with common noise \eqref{eq:MKV} satisfying the following estimate
\begin{equation} \label{eq : estimate MKV solution}
\bigg\| \underset{t \in [0,T]}{\sup} |X_t|\bigg\|_p  \leq C\Big(1+\|X_0\|_p\Big),
\end{equation}
where $C$ is a positive constant depending on $p, T, L, b, \sigma$ and $\sigma^0$. For the proof in the case $p=2$, we refer to Proposition 2.8 in \cite{Carmona_Delarue2}. The proof for $p>2$ follows a similar approach, with only minor differences.

\begin{assumption}
\label{assumption Holder}
    The coefficient mappings $b$, $\sigma$ and $\sigma^0$ are $\rho\,$-Hölder continuous in time, for some $\rho \in (0,1]$ uniformly in space and measure, in the sense that there exists a constant $L>0$ such that for every $s,t\in[0,T]$, $x\in\R^d$, $\mu\in\mathcal{P}_p(\R^d)$, we have
    \begin{equation*}
        \max(|b(t,x,\mu) - b(s,x,\mu)|, |\sigma(t,x,\mu)-\sigma(s,x,\mu)|, |\sigma^0(t,x,\mu)-\sigma^0(s,x,\mu)|)\leq L\big(1+|x|+\mathcal{W}_p(\mu,\delta_0)\big)\:|t-s|^{\rho}.
    \end{equation*}
\end{assumption}

\smallskip
The main results of this paper are the following two theorems whose proofs are presented in Section \ref{section:proof particle method cv}.

\begin{theorem}\label{thm:particle method cv rate}
    Assume that Assumptions \ref{assumption 1}, \ref{assumption Lipschitz} hold for some $p\in[2,\infty)$. Fix $M\in \mathbb{N}$ and set $h=\frac{T}{M}$. For every $m=0, \ldots, M$,  let $\bar{\mu}_{t_m} = \mathcal{L}^1(\bar{X}_{t_m})$, where $\bar{X}_{t_m}$ is defined by \eqref{eq:euler scheme MKV} and let $(\bar{\mu}^N_{t_m})_{1\leq m\leq M}$ denote the empirical measure of the particles $(\bar{X}^{i,N}_{t_m})_{1\leq i \leq N}$ defined by the particle method \eqref{eq:euler particle system}.
    \begin{enumerate}
        \item[(i)] We have
        \[ \underset{m \in \{1,\ldots, M\}}{\sup}\, \big\|  \mathcal{W}_p(\bar{\mu}^N_{t_m}, \bar{\mu}_{t_m})\big\|_p \xrightarrow[N\rightarrow +\infty]{} 0. \]
        \item[(ii)] Moreover, if we assume that Assumption \ref{assumption 1} holds for $p+\varepsilon\in(2,\infty)$ for some $\varepsilon >0$, we have the following rates of convergence
        \begin{equation*}
            \underset{m \in \{1,\ldots, M\}}{\sup} \, \bigg\| \mathcal{W}_p(\bar{\mu}^N_{t_m}, \bar{\mu}_{t_m}) \bigg\|_p \leq C \left\{ 
            \begin{array}{ll}
                N^{-\frac{1}{2p}} + N^{-\frac{\varepsilon}{p(p+\varepsilon)}} & \text{ if } p>\frac{d}{2} \text{ and } \varepsilon\neq p, \\
                N^{-\frac{1}{2p}}\big(\log(1+N)\big)^\frac{1}{p} + N^{-\frac{\varepsilon}{p(p+\varepsilon)}} & \text{ if } p=\frac{d}{2} \text{ and } \varepsilon\neq p,\\
                N^{-\frac{1}{d}}+N^{-\frac{\varepsilon}{p(p+\varepsilon)}} & \text{ if } p\in(0,\frac{d}{2}) \text{ and } p+\varepsilon\neq \frac{d}{d-p},
            \end{array}\right.
        \end{equation*}
        where $C$ is a positive constant which depends on $d$, $p$, $L$, $T$, $\|X_0\|_{p+\varepsilon}$, $b$, $\sigma$ and $\sigma^0$.
    \end{enumerate}
\end{theorem}

\begin{theorem}\label{thm:cv rate scheme}
    Consider $(X_t)_{t\in[0,T]}$ the solution to the McKean-Vlasov equation with common noise \eqref{eq:MKV} and let $(\bar{X}^{1,N}_t)_{t\in[0,T]}$ be the process defined by the particle method \eqref{eq:particle method continuous scheme} which depends on the same Brownian motions as $(X_t)_{t\in[0,T]}$. Then there exists a positive constant $C$ which depends on $d$, $p$, $L$, $T$, $\|X_0\|_{p+\varepsilon}$, $b$, $\sigma$, $\sigma^0$, such that
    \begin{equation} \label{eq:total cv rate}
        \bigg\|\sup_{t\in[0,T]}\big|\bar{X}^{1,N}_t - X_t\big|\bigg\|_p \leq C \big(h^{\frac{1}{2}\wedge \rho} + \mathcal{E}_N\big),
    \end{equation}
    where 
    \begin{equation*}
        \mathcal{E}_N = \left\{
        \begin{array}{ll}
            N^{-\frac{1}{2p}} + N^{-\frac{\varepsilon}{p(p+\varepsilon)}} & \text{ if } p>\frac{d}{2} \text{ and } \varepsilon\neq p, \\
            N^{-\frac{1}{2p}}\big(\log(1+N)\big)^\frac{1}{p} + N^{-\frac{\varepsilon}{p(p+\varepsilon)}} & \text{ if } p=\frac{d}{2} \text{ and } \varepsilon\neq p,\\
            N^{-\frac{1}{d}}+N^{-\frac{\varepsilon}{p(p+\varepsilon)}} & \text{ if } p\in(0,\frac{d}{2}) \text{ and } p+\varepsilon\neq \frac{d}{d-p}.
        \end{array}\right.
    \end{equation*}
\end{theorem}

\smallskip

\begin{remark} \label{rem:comparaison}
We highlight the differences between this paper and \cite{MilsteinSchemeMKV}. In \cite{MilsteinSchemeMKV}, the authors also analyzed the convergence rate of the particle method for the McKean–Vlasov equation with common noise. The key distinctions can be summarized in two aspects: $(i)$ the construction of the numerical approaches, and $(ii)$ the assumptions imposed on the initial random variable and the coefficient functions, along with the resulting convergence rate for the time discretization.

Regarding the first aspect, the approach in \cite{MilsteinSchemeMKV} begins with the particle system used in the conditional propagation of chaos property for the McKean–Vlasov equation (see e.g. \cite[Section 2.1.4]{Carmona_Delarue2}), which is defined by
\begin{align}\label{eq:particle_in_propagation_of_chaos}
&dX_t^{n,N}=b\Big(t,X_t^{n,N}, \frac{1}{N}\sum_{i=1}^{N}\delta_{X_t^{i,N}}\Big)dt+\sigma\Big(t,X_t^{n,N}, \frac{1}{N}\sum_{i=1}^{N}\delta_{X_t^{i,N}}\Big)dW_t^n+\sigma^0\Big(t,X_t^{n,N}, \frac{1}{N}\sum_{i=1}^{N}\delta_{X_t^{i,N}}\Big)dW_t^0\nonumber\\
& \hspace{2.5cm} \mathrm{with } \: X_0^{1,N},..., X_0^{1,N}\widesim{\mathrm{i.i.d.}}X_0, \;W^1, ..., W^N\widesim{\mathrm{i.i.d.}}W,
\end{align}
and subsequently applies a time discretization by using a Milstein-type scheme to this particle system. It is worth noting that \eqref{eq:particle_in_propagation_of_chaos} can be regarded as a high-dimensional equation in terms of state arguments, without involving the measure argument. This feature enables the application of classical numerical analysis methods for diffusion processes to study the system \eqref{eq:particle_in_propagation_of_chaos}. In our paper, we first apply time discretization using the Euler scheme, retaining the measure argument within the scheme. This approach allows for the potential integration of other spatial discretization methods in future work, such as the optimal quantization method, as discussed in \cite{NumMethMKV} for the standard McKean–Vlasov equation without common noise. 

As for the second aspect, the approach in \cite{MilsteinSchemeMKV} imposes additional regularity conditions on the coefficient functions $\sigma$ and $\sigma^0$ with respect to both the state and measure arguments. Specifically, it requires:
\begin{align}
&|\partial_x \sigma^u_{\ell}(t, x, \mu) \sigma^v_{\ell}(t, x, \mu)-\partial_x \sigma^u_{\ell}(t, x^{\prime}, \mu^{\prime}) \sigma^v_{\ell}(t, x^{\prime}, \mu^{\prime})|  \leq L\{|x-x^{\prime}|+\mathcal{W}_2(\mu, \mu^{\prime})\}, \nonumber\\
&|\partial_\mu \sigma^u_{\ell}(t, x, \mu, y) \sigma^v_{\ell}(t, y, \mu)-\partial_\mu \sigma^u_{\ell}(t, x^{\prime}, \mu^{\prime}, y^{\prime}) \sigma^v_{\ell}(t, y^{\prime}, \mu^{\prime})|  \leq L\{|x-x^{\prime}|+|y-y^{\prime}|+\mathcal{W}_2(\mu, \mu^{\prime})\},\nonumber
\end{align}
for all $u, v \in\{0,1\}, \ell \in\{1, \ldots, m_u\}, \ell_1 \in\{1, \ldots, m_v\}, t \in[0, T], x, x^{\prime}, y, y^{\prime} \in \mathbb{R}^d$ and $\mu, \mu^{\prime} \in \mathcal{P}_2(\mathbb{R}^d)$, where $\sigma^1$ in their paper corresponds to $\sigma$ here. Additionally, they assume that 
$X_0$ has a finite $p$-th moment with $p\geq4$. These conditions enable a faster convergence rate with respect to the time step 
$h$. However, they exclude certain coefficient functions, such as
 $x\mapsto \sigma(x)=|x|$ or $x \in \R \longmapsto \sigma(x)=\sqrt{x}\mathds{1}_{[0,1]}(x) + \mathds{1}_{(1,\infty]}(x)$, which can be handled within the framework proposed in this paper.
\end{remark}

\subsection{Organization of the paper}

The paper is organized as follows. Some preliminary results are gathered in Section \ref{section : prob settings} along with some notations. Sections \ref{section:euler scheme cv} and \ref{section:proof particle method cv} respectively present the proofs for the convergence rates of the Euler scheme (see further Proposition \ref{prop:euler scheme cv}) and the particle method (Theorem \ref{thm:particle method cv rate}, Theorem \ref{thm:cv rate scheme}). Section \ref{section : simulation examples} provides numerical examples to illustrate the methods discussed in this paper. The first example is a modified Ornstein-Uhlenbeck process to which Brownian common noise has been added. For the second example, we simulate the Interbank market model presented in \cite{ren2024risk}[Section 5] which is an application of a risk-sensitive mean field games with common noise. Appendix \ref{section:appendix} is dedicated to presenting the detailed proofs of the lemmas referenced throughout the paper, that are essential to supporting the proofs of the main results.

\medskip
\section{Preliminary results} \label{section : prob settings}

In this paper, we fix a terminal time $T>0$, and denote the space of continuous function from $[0,T]$ to a Polish space $S$ by $\mathcal{C}([0,T],S)$. We also use the notation $\mathbb{L}^p([0,T]\times\Omega)$ for the set of $(\mathcal{F}_t)_{t\in[0,T]}$-progressively measurable continuous processes $X=(X_t)_{t\in[0,T]}$ such that 
\[\|X\|_{\mathbb{L}^p([0,T]\times\Omega)} =\bigg\| \underset{t \in [0,T]}{\sup} |X_t|\bigg\|_p < \infty.\]

We now list some key lemmas that will support the subsequent proofs.
\begin{lemma}[Lemma 2.5 in \cite{Carmona_Delarue2}] \label{lem: 2.5 Carmona}
    Given an $\R^d$-valued process $(X_t)_{t\in[0,T]}$, adapted to the filtration $\mathbb{F}$, consider for any $t\in[0,T]$, a version of $\mathcal{L}^1(X_t)$ as defined in \eqref{eq:def-L1X}. Then, the $\mathcal{P}(\R^d)$-valued process $(\mathcal{L}^1(X_t))_{t\in[0,T]}$ is adapted to $\mathbb{F}^0$. If, moreover, $(X_t)_{t\in[0,T]}$ has continuous paths and satisfies $\E\big[\underset{0\leq t\leq T}{\sup} |X_t|^p\big] < \infty$, then we can find a version of each $\mathcal{L}^1(X_t)$, $t\in[0,T]$, such that the process $(\mathcal{L}^1(X_t))_{t\in[0,T]}$ has continuous paths in $\mathcal{P}_p(\R^d)$ and is $\mathbb{F}^0$-adapted.
\end{lemma}
Consider now the unique strong solution $X=(X_t)_{t\in[0,T]}$ of \eqref{eq:MKV}. Lemma \ref{lem: 2.5 Carmona} above, and  Proposition 2.9, Remark 2.10 in \cite{Carmona_Delarue2} imply that there exists a version of each $\mathcal{L}^1(X_t)$ , $t\in[0,T]$, such that the process $(\mathcal{L}^1(X_t))_{t\in[0,T]}$ has continuous paths in $\mathcal{P}_p(\R^d)$ and that it provides a version of the conditional law of $X_t$ given $W^0$.

\begin{remark}[Remark 2.3 in \cite{Carmona_Delarue2}]
    With a slight abuse of notation, we shall not distinguish a random variable $X$ constructed on $(\Omega^0, \mathcal{F}^0, \mathbb{P}^0)$ (resp. $(\Omega^1, \mathcal{F}^1, \mathbb{P}^1)$) with its natural extension $\tilde{X} : (\omega^0,\omega^1) \mapsto X(\omega^0)$ (resp. $\tilde{X} : (\omega^0,\omega^1) \mapsto X(\omega^1)$) on $(\Omega, \mathcal{F}, \mathbb{P})$. Similarly, for a sub-$\sigma$-algebra $\mathcal{G}^0$ of $\mathcal{F}^0$ (resp. $\mathcal{G}^1$ of $\mathcal{F}^1$), we shall often just write $\mathcal{G}^0$ (resp. $\mathcal{G}^1$) for the sub-$\sigma$-algebra $\mathcal{G}^0\otimes \{\emptyset,\,\Omega^1\}$ (resp.$\{\emptyset,\,\Omega^0\} \otimes \mathcal{G}^1$).
\end{remark}

\begin{lemma}[General Minkowski inequality]
\label{lem:general minkowski}
    For every $p\in[1,\infty)$, for every process $(X_t)_{t\in[0,T]}$ and for every $T>0$,
    \begin{equation*}
        \bigg\|\int_0^T X_t dt\bigg\|_p\leq\int_0^T\|X_t\|_p dt.
    \end{equation*}
\end{lemma}

\begin{lemma}[Burkölder-Davis-Gundy inequality]
\label{lem:BDG} 
    For every $p\in(0,\infty)$, there exist two positive constants $c_p$, $C_p$ such that, for every continuous local martingale $X=(X_t)_{t\in[0,T]}$ which vanishes at 0,
    \[ c_p\big\|\big(\langle X\rangle_T\big)^\frac{1}{2}\big\|_p\leq \bigg\| \underset{0\leq t\leq T}{\sup}|X_t|\bigg\|_p\leq C_p \big\|\big(\big\langle X\rangle_T\big)^\frac{1}{2}\big\|_p.\]
\end{lemma}

\begin{lemma}[`\textit{A la Gronwall}' Lemma]
\label{lem:a la Gronwall}
    Let $f:[0,T]\longrightarrow \R_+$ be a Borel, locally bounded, non-negative and non-decreasing mapping, let $g:[0,T]\longrightarrow\R_+$ be a non-negative and non-decreasing mapping such that:
    \[f(t)\leq C_1\int_0^tf(s)ds + C_2\Big(\int_0^tf^2(s)sds\Big)^\frac{1}{2} +g(t), \quad \forall t\in[0,T], \]
    where $C_1$ and $C_2$ are two positive constants. Then for any $t\in[0,T]$,
    \[f(t)\leq 2\exp\{(2C_1+C_2^2)t\} g(t). \]
\end{lemma}
We refer to Section 7.8 in \cite{NumericalProbability} for the proofs of Lemma \ref{lem:general minkowski}, Lemma \ref{lem:BDG} and Lemma \ref{lem:a la Gronwall}. 
\noindent
Moreover, we have the following result on the $p$-Wasserstein distance between conditional laws, whose proof is postponed to Appendix \ref{section:appendix}. 
\begin{lemma}
\label{lem:cond expec inequality}
    For $Y_1, Y_2$ two random variables on $(\Omega, \mathcal{F}, \mathbb{P})$ with a finite $p$-th moment, $p\in[1,\infty)$, we have
    \[\Big\| \mathcal{W}_p\big(\mathcal{L}^1(Y_1), \mathcal{L}^1(Y_2)\big) \Big\|_p \leq \big\|Y_1 - Y_2 \big\|_p. \]
\end{lemma}

\begin{remark}
    From Lemma \ref{lem:cond expec inequality}, we deduce that for every random variable $X\in L^p(\Omega)$, we have
    \[\big\|\mathcal{W}_p\big(\mathcal{L}^1(X),\delta_0\big)\big\|_p \leq \|X\|_p. \]
\end{remark}

\medskip

\section{Convergence rate of the Euler scheme} \label{section:euler scheme cv}

In this section, we prove the convergence rate of the Euler scheme as described in the following proposition.

\begin{proposition} \label{prop:euler scheme cv}
    Let $X=(X_t)_{t\in[0,T]}$ denote the unique solution of the McKean-Vlasov equation \eqref{eq:MKV}, and let $\bar{X}^M = (\bar{X}^M_t)_{t\in[0,T]}$ denote the process defined by the continuous Euler scheme \eqref{eq:continuous euler scheme MKV}. Under Assumptions \ref{assumption 1}, \ref{assumption Lipschitz} and \ref{assumption Holder}, there exists a constant $C$ depending on $d$, $p$, $L$, $T$, $\rho$ and $\|X_0\|_{p}$ such that
    \begin{equation*}
        \big\|X-\bar{X}^M\big\|_{\mathbb{L}^p} = \bigg\| \underset{0\leq t \leq T}{\sup}\big|X_t-\bar{X}^M_t\big| \bigg\|_p \leq Ch^{\frac{1}{2}\wedge \rho}.
    \end{equation*}
\end{proposition}

The proof of Proposition \ref{prop:euler scheme cv} needs the two following lemmas whose proofs are postponed to Appendix \ref{section:appendix}. In this paper, a constant denoted by $C_{p_1,\, \ldots, p_n}$ is a constant depending on parameters $p_1,\, \ldots, p_n$, whose value can change line ton line. 

\begin{lemma}
    \label{lem:Lp euler scheme}
    Let $(\bar{X}^M_t)_{t\in[0,T]}$ be the process defined by the continuous extension of the Euler scheme \eqref{eq:continuous euler scheme MKV}. Assume Assumptions \ref{assumption 1} and \ref{assumption Lipschitz} hold. Then, for every $M\geq 1$, $(\bar{X}^M_t)_{t\in[0,T]} \in\mathbb{L}^p([0,T]\times\Omega)$ and  there exists a non-negative constant $C$ such that 
    \[\bigg\|\underset{0\leq t\leq T}{\sup} \big|\bar{X}^M_t\big|\bigg\|_p \leq C\big(1+\big\|X_0\big\|_p\big), \]
    where $C$ depends on $p, T, L$ and the coefficients $b,\sigma$ and $\sigma^0$.
\end{lemma}

The following lemma establishes that $\mathcal{L}^1(\bar{X}_t)$ is a version of the conditional law of $\bar{X}_t$ given the common noise $W^0$, whose proof is postponed to Appendix \ref{section:appendix}.

\begin{lemma} \label{lem:L1_cond_law}
Assume that Assumption \ref{assumption 1} holds with some $p\in[2, +\infty)$. Let $\bar{X}^M=(\bar{X}_t^M)_{t\in[0,T]}$ be the process defined by the continuous Euler scheme \eqref{eq:continuous euler scheme MKV}. Then for any $t\in[0,T]$,  $\mathcal{L}^1(\bar{X}_t)$ provides a  version of the conditional distribution of $\bar{X}_t$ given $W^0$. Moreover, we can find a version of $(\mathcal{L}^1(\bar{X}_t))_{t\in[0,T]}$ such that it is $\mathbb{F}^0$-adapted and has continuous paths in $\mathcal{P}_p(\R^d)$.
\end{lemma}

In the subsequent discussion, to simplify notation, we directly denote by $(\mathcal{L}^1(X_t))_{t\in[0,T]}$ and  $(\mathcal{L}^1(\bar{X}_t))_{t\in[0,T]}$ the versions with continuous paths in $\mathcal{P}_p(\R^d)$.

\begin{lemma}
\label{lem:holder continuity solution}
 Consider $X=(X_t)_{t\in[0,T]}$ the solution to the McKean-Vlasov equation with common noise \eqref{eq:MKV}. Assume Assumption \ref{assumption 1} and \ref{assumption Lipschitz} hold. For every $0\leq s\leq t\leq T$, we have
 \[ \big\|X_t-X_s\big\|_p \leq C (t-s)^\frac{1}{2}, \]
 where the constant $C$ depends on $p$, $L$, $T$ and the data $(\|X_0\|_p, b, \sigma, \sigma^0)$.
\end{lemma}

\smallskip

\noindent
For the sake of clarity in the proof of Proposition \ref{prop:euler scheme cv}, we introduce the following notation, for every $m\in\{0,...,M-1\}$ and for every $t\in[t_m, t_{m+1})$, we define 
\begin{equation} \label{eq:time discretization}
    \underline{t}\coloneqq t_m.
\end{equation}

\begin{proof}[Proof of Proposition \ref{prop:euler scheme cv}]

Denote $(\mu_t)_{t\in[0,T]} = \big(\mathcal{L}^1(X_t)\big)_{t\in[0,T]}$ and $(\bar{\mu}^M_t)_{t\in[0,T]} = \big(\mathcal{L}^1(\bar{X}^M_t) \big)_{t\in[0,T]}$ that is well-defined by Lemma \ref{lem:L1_cond_law}. We write $\bar{X}_s$ and $\bar{\mu}_s$ instead of $\bar{X}^M_s$ and $\bar{\mu}^M_s$ when there is no ambiguity. Using Inequality \eqref{eq : estimate MKV solution} and Lemma \ref{lem:Lp euler scheme}, we deduce that $X=(X_t)_{t\in[0,T]},\, (\bar{X}^M_t)_{t\in[0,T]}$ belongs to $\mathbb{L}^p([0,T]\times\Omega)$. Consequently $(\mu_t)_{t\in[0,T]}$ and $(\bar{\mu}^M_t)_{t\in[0,T]}$ take values in $\mathcal{C}\big([0,T], \mathcal{P}_p(\R^d)\big)$. Fix $t\in[0,T]$, by Minkowski's inequality, we get that
\begin{equation} \label{eq:bound euler scheme error}
\begin{aligned}
    \bigg\| \underset{0\leq u \leq t}{\sup}\big|X_u-\bar{X}_u^M\big| \bigg\|_p & \leq \bigg\|\underset{0\leq u \leq t}{\sup} \int_0^u \Big(b\big(s,X_s,\mu_s\big) - b(\underline{s},\bar{X}_{\underline{s}}, \bar{\mu}_{\underline{s}})\Big)ds\bigg\|_p\\ 
    &\quad + \bigg\|\underset{0\leq u \leq t}{\sup} \int_0^u \Big(\sigma(s,X_s,\mu_s)- \sigma(\underline{s},\bar{X}_{\underline{s}}, \bar{\mu}_{\underline{s}})\Big)dW_s\bigg\|_p\\ 
    &\quad + \bigg\| \underset{0\leq u \leq t}{\sup} \int_0^u\Big(\sigma^0(s,X_s,\mu_s)- \sigma^0(\underline{s},\bar{X}_{\underline{s}}, \bar{\mu}_{\underline{s}})\Big)dW^0_s\bigg\|_p.
\end{aligned}
\end{equation}
In the following proof, we provide an upper bound for each term of the right-hand side of \eqref{eq:bound euler scheme error}. For the first term, by general Minkowski's inequality (Lemma \ref{lem:general minkowski}), we have 
\begin{align*}
    \bigg\|\underset{0\leq u \leq t}{\sup}& \int_0^u \Big(b\big(s,X_s,\mu_s\big) - b(\underline{s},\bar{X}_{\underline{s}}, \bar{\mu}_{\underline{s}})\Big)ds\bigg\|_p \leq \bigg\| \int_0^t \Big|b(s,X_s,\mu_s) - b(\underline{s},\bar{X}_{\underline{s}}, \bar{\mu}_{\underline{s}})\Big|ds\bigg\|_p  \\
    &\leq \int_0^t \Big\|b(s,X_s,\mu_s) - b(\underline{s},\bar{X}_{\underline{s}}, \bar{\mu}_{\underline{s}})\Big\|_pds  \\
    &\leq \int_0^t \Big\|b(s,X_s,\mu_s) - b(\underline{s},X_s,\mu_s)\Big\|_pds + \int_0^t \Big\|b(\underline{s},X_s,\mu_s\big) - b(\underline{s},\bar{X}_{\underline{s}}, \bar{\mu}_{\underline{s}})\Big\|_pds  \\
    &\leq L\int_0^t \Big\|(s-\underline{s})^\rho \big(1+|X_s|+\mathcal{W}_p(\mu_s,\delta_0)\big)\Big\|_pds + L\int_0^t \Big\|\big|X_s- \bar{X}_{\underline{s}}\big|+\mathcal{W}_p(\mu_s,\bar{\mu}_{\underline{s}})\Big\|_pds,
\end{align*}
where we use the $\rho\,$-Hölder and $L\,$-Lipschitz continuity of $b$. From the estimate \eqref{eq : estimate MKV solution} of the solution process $(X_t)_{t\in[0,T]}$, we can deduce that

\begin{align}\label{eq:bound b euler}
    \bigg\|\underset{0\leq u \leq t}{\sup}& \int_0^u \Big(b\big(s,X_s,\mu_s\big) - b(\underline{s},\bar{X}_{\underline{s}}, \bar{\mu}_{\underline{s}})\Big)ds\bigg\|_p \notag \\
    &\leq L \underset{0\leq u\leq T}{\sup}|u-\underline{u}|^\rho \int_0^t \big(1+\|X_s\|_p + \|\mathcal{W}_p(\mu_s,\delta_0)\|_p\big)ds + L\int_0^t \big(\|X_s-\bar{X}_{\underline{s}}\|_p + \big\|\mathcal{W}_p(\mu_s,\bar{\mu}_{\underline{s}})\big\|_p\big)ds \notag \\
    &\leq Lh^\rho \int_0^t \big(1+2\|X_s\|_p\big)ds + 2L\int_0^t \|X_s-\bar{X}_{\underline{s}}\|_pds \notag  \\
    &\leq LTh^\rho \bigg(1+\bigg\|\underset{0\leq s\leq t}{\sup}\big|X_s\big|\bigg\|_p\bigg) + 2L\int_0^t \big(\|X_s-X_{\underline{s}}\|_p + \|X_{\underline{s}}-\bar{X}_{\underline{s}}\|_p\big)ds \notag \\
    &\leq C_{p,T,L,b,\sigma,\sigma^0}\big(1+\|X_0\|_p\big) h^\rho + 2L\int_0^t \Big(\|X_s-X_{\underline{s}}\|_p + \bigg\| \underset{0\leq u\leq s}{\sup}\big|X_{\underline{u}}-\bar{X}_{\underline{u}}\big|\bigg\|_p\Big)ds.
\end{align}
We apply Lemma \ref{lem:holder continuity solution} to Inequality \eqref{eq:bound b euler} to derive that
\begin{align}\label{eq:final bound for b, euler}
     \bigg\|\underset{0\leq u \leq t}{\sup}& \int_0^u \Big(b\big(s,X_s,\mu_s\big) - b(\underline{s},\bar{X}_{\underline{s}}, \bar{\mu}_{\underline{s}})\Big)ds\bigg\|_p \notag \\ &\leq C_{p,T,L,b,\sigma,\sigma^0}\big(1+\|X_0\|_p\big) h^\rho + L\int_0^t \Big(\|X_s-X_{\underline{s}}\|_p + \bigg\| \underset{0\leq u\leq s}{\sup}\big|X_u-\bar{X}_{u}\big|\bigg\|_p\Big)ds \notag \\
     &\leq C_{p,T,L,b,\sigma,\sigma^0}\big(1+\|X_0\|_p\big) h^\rho  + LT\kappa\underset{0\leq s\leq T}{\sup} |s-\underline{s}|^\frac{1}{2} + \int_0^t \bigg\| \underset{0\leq u\leq s}{\sup} \big|X_u-\bar{X}_u\big|\bigg\|_pds \notag \\
     &\leq C_{p,T,L,b,\sigma,\sigma^0, \|X_0\|_p}\: h^{\frac{1}{2}\wedge\rho} + \int_0^t \bigg\|\underset{0\leq u\leq s}{\sup}\big|X_u-\bar{X}_u\big|\bigg\|_pds.
\end{align}

For the second and the last terms of Inequality \eqref{eq:bound euler scheme error}, the computations are very similar since $\sigma$ and $\sigma^0$ have the same regularity as that of $b$. By the Burkölder-Davis-Gundy inequality (Lemma \ref{lem:BDG}), we have
\begin{align}
    \bigg\|&\underset{0\leq u \leq t}{\sup} \int_0^u \Big(\sigma(s,X_s,\mu_s)- \sigma(\underline{s},\bar{X}_{\underline{s}}, \bar{\mu}_{\underline{s}})\Big)dW_s\bigg\|_p \notag  \\ &\leq C_{d,p} \bigg\| \Big(\int_0^t \big|\sigma(s,X_s,\mu_s)-\sigma(\underline{s},\bar{X}_{\underline{s}},\bar{\mu}_{\underline{s}})\big|^2ds\Big)^\frac{1}{2}\bigg\|_p  
    =C_{d,p}\bigg\| \int_0^t \big|\sigma(s,X_s,\mu_s)-\sigma(\underline{s},\bar{X}_{\underline{s}},\bar{\mu}_{\underline{s}})\big|^2ds\bigg\|_\frac{p}{2}^\frac{1}{2} \notag  \\
    &\leq C_{d,p}\bigg(\int_0^t \Big\|\big| \sigma(s,X_s,\mu_s)-\sigma(\underline{s},\bar{X}_{\underline{s}},\bar{\mu}_{\underline{s}})\big|^2\Big\|_\frac{p}{2}ds\bigg)^\frac{1}{2} 
    = C_{d,p}\bigg(\int_0^t \Big\|\sigma(s,X_s,\mu_s)-\sigma(\underline{s},\bar{X}_{\underline{s}},\bar{\mu}_{\underline{s}})\Big\|_p^2ds\bigg)^\frac{1}{2} \notag \\
    &\leq C_{d,p}\bigg(\int_0^t \Big\|\sigma(s,X_s,\mu_s)-\sigma(\underline{s},X_s,\mu_s)\Big\|_p^2ds + \int_0^t \Big\|\sigma(\underline{s},X_s,\mu_s)-\sigma(\underline{s},\bar{X}_{\underline{s}},\bar{\mu}_{\underline{s}})\Big\|_p^2ds\bigg)^\frac{1}{2}, \notag
\end{align}
where we used the Minkowski's and Young's inequalities. Due to the $\rho\,$-Hölder and $L\,$-Lipschitz continuity of the mapping $\sigma$, we obtain that

\begin{align}
    \bigg\|&\underset{0\leq u \leq t}{\sup} \int_0^u \Big(\sigma(s,X_s,\mu_s)- \sigma(\underline{s},\bar{X}_{\underline{s}}, \bar{\mu}_{\underline{s}})\Big)dW_s\bigg\|_p \notag \\ 
    &\leq C_{d,p}\bigg(\int_0^t \Big\|(s-\underline{s})^{2\rho}\big(1+|X_s|+\mathcal{W}_p(\mu_s,\delta_0)\big)\Big\|_p^2ds + L^2\int_0^t \Big\|\big|X_s-\bar{X}_{\underline{s}}\big| + \mathcal{W}_p(\mu_s,\bar{\mu}_{\underline{s}})\Big\|_p^2ds\bigg)^\frac{1}{2}\notag \\
    &\leq C_{d,p} \bigg[\underset{0\leq u\leq T}{\sup}|s-\underline{s}|^{2\rho} \! \int_0^t \! \Big(1+\|X_s\|_p  +\big\|\mathcal{W}_p(\mu_s,\delta_0)\big\|_p\Big)ds  + 2L^2\!\int_0^t \! \Big(\big\|X_s-\bar{X}_{\underline{s}}\big\|_p^2 + \big\| \mathcal{W}_p(\mu_s,\bar{\mu}_{\underline{s}}) \big\|_p^2\Big)ds\bigg]^\frac{1}{2} \notag \\
    &\leq C_{d,p} \bigg(h^{2\rho }\int_0^t \bigg(1 + 2\bigg\|\underset{0\leq u\leq s}{\sup}\big|X_u\big|\bigg\|_p\bigg)ds + 4L^2\int_0^t \big\|X_s-\bar{X}_{\underline{s}}\big\|_p^2ds\bigg)^\frac{1}{2}. \notag 
\end{align}
Applying the estimate \eqref{eq : estimate MKV solution} yields to

\begin{align}
    \bigg\|&\underset{0\leq u \leq t}{\sup} \int_0^u \Big(\sigma(s,X_s,\mu_s)- \sigma(\underline{s},\bar{X}_{\underline{s}}, \bar{\mu}_{\underline{s}})\Big)dW_s\bigg\|_p \notag \\ 
    &\leq C_{d,p}\bigg(C_{p,L,T,b,\sigma,\sigma^0}\big(1+ \|X_0\|_p\big)^2h^{2\rho} +  8L^2 \int_0^t \Big(\big\|X_s-X_{\underline{s}}\big\|_p^2 +\big\|X_{\underline{s}}-\bar{X}_{\underline{s}}\big\|_p^2\Big)ds\bigg)^\frac{1}{2} \notag \\
    &\leq C_{d,p,L,T,b,\sigma,\sigma^0}\bigg[\big(1+\|X_0\|_p\big)h^\rho
    + \bigg(\int_0^t\Big(\big\|X_s-X_{\underline{s}}\big\|_p^2 + \bigg\| \underset{0\leq u\leq s}{\sup}\big|X_u-\bar{X}_u\big| \bigg\|_p^2\Big)ds\bigg)^\frac{1}{2}\bigg]. \label{eq:bound sigma euler}
 \end{align}
 \noindent   
Applying again Lemma \ref{lem:holder continuity solution} to Inequality \eqref{eq:bound sigma euler} yields to
 \begin{align}\label{eq:final bound for sigma, euler}
     &\bigg\|\underset{0\leq u \leq t}{\sup} \int_0^u \Big(\sigma(s,X_s,\mu_s)- \sigma(\underline{s},\bar{X}_{\underline{s}}, \bar{\mu}_{\underline{s}})\Big)dW_s\bigg\|_p \notag \\
     &\leq C_{d,p,L,T,b,\sigma,\sigma^0} \bigg[\big(1+\|X_0\|_p\big)h^\rho
    + \bigg(\int_0^t\bigg(\big\|X_s-X_{\underline{s}}\big\|_p^2 + \bigg\| \underset{0\leq u\leq s}{\sup}\big|X_u-\bar{X}_u\big| \bigg\|_p^2\bigg)ds\bigg)^\frac{1}{2}\bigg] \notag \\
    &\leq C_{d,p,L,T,b,\sigma,\sigma^0} \bigg[\big(1+\|X_0\|_p\big)h^\rho
    + \underset{0\leq s\leq T}{\sup}|s-\underline{s}|^\frac{1}{2} + \bigg(\int_0^t\bigg\| \underset{0\leq u\leq s}{\sup}\big|X_u-\bar{X}_u\big|\bigg\|_p^2ds\bigg)^\frac{1}{2}\bigg] \notag \\
    &\leq C_{d,p,L,T,b,\sigma,\sigma^0,\|X_0\|_p} h^{\frac{1}{2}\wedge \rho} + \bigg(\int_0^t\bigg\| \underset{0\leq u\leq s}{\sup}\big|X_u-\bar{X}_u\big|\bigg\|_p^2ds\bigg)^\frac{1}{2}.
 \end{align}
 We can repeat the same reasoning for $\sigma^0$ and $W^0$ in order to deduce a similar upper bound
 \begin{equation} \label{eq:final bound for sigma0, euler}
 \begin{aligned}
     \bigg\|\underset{0\leq u \leq t}{\sup} \int_0^u &\Big(\sigma^0(s,X_s,\mu_s)- \sigma^0(\underline{s},\bar{X}_{\underline{s}}, \bar{\mu}_{\underline{s}})\Big)dW^0_s\bigg\|_p \\ 
     &\leq C_{d,p,L,T,b,\sigma,\sigma^0,\|X_0\|_p} h^{\frac{1}{2}\wedge \rho} + \bigg(\int_0^t\bigg\| \underset{0\leq u\leq s}{\sup}\big|X_u-\bar{X}_u\big|\bigg\|_p^2ds\bigg)^\frac{1}{2}.
 \end{aligned}
 \end{equation}
 We plug the Inequalities \eqref{eq:final bound for b, euler}, \eqref{eq:final bound for sigma, euler} and \eqref{eq:final bound for sigma0, euler} into Inequality \eqref{eq:bound euler scheme error} to get that
 \begin{equation} \label{eq:final bound Lp error euler}
 \begin{aligned}
 \bigg\|\underset{0\leq s\leq t}{\sup}\big|X_s - \bar{X}_s\big|\bigg\|_p &\leq \int_0^t \bigg\|\underset{0\leq u\leq s}{\sup}\big|X_u-\bar{X}_u\big|\bigg\|_pds +\bigg(\int_0^t\bigg\|\underset{0\leq u\leq s}{\sup}\big|X_u-\bar{X}_u\big|\bigg\|_p^2ds\bigg)^\frac{1}{2} \\ & \quad + C_{d,p,L,T,b,\sigma,\sigma^0,\|X_0\|_p} h^{\frac{1}{2}\wedge\rho}.
 \end{aligned}
 \end{equation}
 \smallskip
 Since $X=(X_t)_{t\in[0,T]}$ and $(\bar{X}^M_t)_{t\in[0,T]}$ belongs to $\mathbb{L}^p$, the application 
 \[ t \longmapsto \bigg\|\underset{0\leq s\leq t}{\sup}\big|X_s - \bar{X}_s\big|\bigg\|_p \]
 is continuous, non-decreasing and non-negative on $[0,T]$. We conclude this proof by applying Lemma \ref{lem:a la Gronwall} to Inequality \eqref{eq:final bound Lp error euler} and deduce the existence of a constant $C$ depending on the parameters $d$, $p$, $L$, $T$, and the data $(\|X_0\|_p,b,\sigma,\sigma^0)$ such that we have
 \begin{equation*}
     \bigg\|\underset{0\leq s\leq t}{\sup}\big|X_s - \bar{X}_s\big|\bigg\|_p \leq Ch^{\frac{1}{2}\wedge\rho}. \hfill \qedhere
 \end{equation*}
\end{proof}

\medskip
\section{Convergence rate of the particle system} \label{section:proof particle method cv}

This section is devoted to proving the convergence rate of the particle method, as described in Theorem \ref{thm:particle method cv rate} and Theorem \ref{thm:cv rate scheme}. To do this, we need the following $N$-particle system $(\bar{Y}^1, \ldots, \bar{Y}^N)$ without interaction. Recall the definition of $\underline{s}$ in \eqref{eq:time discretization}. Let $W^0,W^1,\, \ldots, W^N$ be the same Wiener processes as defined in \eqref{eq:particle method continuous scheme}. 

\noindent
\begin{equation}\label{eq:continuous euler solution scheme}
 \forall\,1\leq i\leq N,\quad   \bar{Y}^{i}_{t} = X^i_0 + \int_0^tb\big(\underline{s}, \bar{Y}^{i}_{\underline{s}}, \mathcal{L}^1(\bar{Y}^i_{\underline{s}})\big)ds + \int_0^t \sigma\big(\underline{s}, \bar{Y}^{i}_{\underline{s}}, \mathcal{L}^1(\bar{Y}^i_{\underline{s}})\big)dW^i_s + \int_0^t \sigma^0\big(\underline{s}, \bar{Y}^{i}_{\underline{s}}, \mathcal{L}^1(\bar{Y}^i_{\underline{s}})\big)dW^0_s.
\end{equation}
We have the following property of the particles in the system \eqref{eq:continuous euler solution scheme}, whose proof is postponed to Appendix \ref{section:appendix}.
\begin{lemma}
\label{lem:MKV sol cond iid}
    The particles $\bar{Y}^1,\, \ldots, \bar{Y}^N$ are identically distributed having the same distribution as $\bar{X}$ defined by the continuous Euler scheme \eqref{eq:continuous euler scheme MKV}, and independent conditionally to $W^0$.
\end{lemma}

\noindent
Hence, we may still use the same notation $\bar{\mu}_t$ for $\mathcal{L}^1(\bar{Y}^i_{t})$ in the proof. In order to prove Theorem \ref{thm:particle method cv rate}, we will need the following results in addition of Lemma \ref{lem:MKV sol cond iid} (see \cite[Corollary 2.14]{Lacker} and \cite[Theorem 1]{FournierGuillin} for the proof).

\begin{lemma}[Corollary 2.14 in \cite{Lacker}]
\label{lem:Wp convergence empirical measure}
    Suppose $(X_i)_{1\leq i\leq N}$ are i.i.d. $\R^d$-valued random variables with law $\mu$ and let $\mu_N = \frac{1}{N}\sum_{i=1}^N\delta_{X_i}$ denote the empirical measure. If $\mu \in \mathcal{P}_p(\R^d)$, $p\geq 1$, then $\mathcal{W}_p(\mu_N,\mu) \rightarrow 0$ almost surely, and also $\E\big[\mathcal{W}_p^p(\mu_N,\mu) \big] \rightarrow{} 0$ when $N\rightarrow +\infty$.
\end{lemma}

\begin{lemma}[Theorem 1 in \cite{FournierGuillin}]
\label{lem:FournierGuillin}
    Let $\mu\in\mathcal{P}(\R^d)$. Assume that $M_q(\mu):=\int_{\R^d}|x|^q\mu(dx)<\infty$ for some $q>p$. We consider an i.i.d. sequence $(X_k)_{k\geq1}$ of $\mu$-distributed random variables and, for $N\geq 1$, the empirical measure $\mu_N := \frac{1}{N}\sum_{k=1}^N\delta_{X_k}$. There exists a positive constant $C_{d,p,q}$ such that, for every $N\geq 1$, we have
    \begin{equation*}
        \E\big[\mathcal{W}_p^p(\mu_N,\mu)\big] \leq C_{d,p,q}M^{p/q}_q(\mu)\left\{
        \begin{array}{ll}
        N^{-\frac{1}{2}} + N^{-(q-p)/q} & \text{ if } p>\frac{d}{2} \text{ and } q\neq2p,  \\
        N^{-\frac{1}{2}}\log(1+N) + N^{-(q-p)/q}  & \text{ if } p=\frac{d}{2} \text{ and } q\neq2p, \\
        N^{-\frac{p}{d}}+N^{-(q-p)/q} & \text{ if } p\in(0,\frac{d}{2}) \text{ and } q\neq \frac{d}{d-p}.
        \end{array}\right.
    \end{equation*}
\end{lemma}

\begin{remark}\label{rq:cv rate PoC}
    We deduce from Lemma \ref{lem:FournierGuillin} that
    \begin{equation*}
        \big\|\mathcal{W}_p(\mu_N,\mu)\big\|_p \leq C_{d,p,q} M^{1/q}_q(\mu)\left\{
        \begin{array}{ll}
        N^{-\frac{1}{2p}} + N^{-(q-p)/qp} & \text{ if } p>\frac{d}{2} \text{ and } q\neq2p,  \\
        N^{-\frac{1}{2p}}\log(1+N)^\frac{1}{p} + N^{-(q-p)/qp}  & \text{ if } p=\frac{d}{2} \text{ and } q\neq2p, \\
        N^{-\frac{1}{d}}+N^{-(q-p)/qp} & \text{ if } p\in(0,\frac{d}{2}) \text{ and } q\neq \frac{d}{d-p}.
        \end{array}\right.
    \end{equation*}
\end{remark}

\begin{proof}[Proof of Theorem \ref{thm:particle method cv rate}]
\textit{(i)} Recall that, for every $t\in[0,T]$, $\bar{\mu}_t = \mathcal{L}^1(\bar{X}_t)$, $\bar{\mu}^N_t$ is the empirical distribution of the particles $(\bar{X}^{i,N}_t)_{1\leq i\leq N}$ defined in \eqref{eq:particle method continuous scheme}. Using the Minkowski inequality, we can bound $\underset{m \in \{1,\ldots, M\}}{\sup} \big\| \mathcal{W}_p(\bar{\mu}^N_{t_m},\bar{\mu}_{t_m})\big\|_p$ by two terms, which we later demonstrate that they converge to 0 at the desired convergence rate. Consider $(\bar{Y}^1,\dots, \bar{Y}^N)$ identically distributed copies of the continuous Euler scheme $\bar{X}$ defined in \eqref{eq:continuous euler solution scheme}. For every $t\in[0,T]$, we define $\nu^N_t = \frac{1}{N}\sum_{i=1}^N \delta_{\bar{Y}^i_t}$. It follows that, for every $m \in \{1,\ldots, M\}$,
\begin{align}
    \underset{k \in \{1,\ldots, m\}}{\sup} \big\| \mathcal{W}_p(\bar{\mu}^N_{t_k},\bar{\mu}_{t_k})\big\|_p
    &\leq \underset{k \in \{1,\ldots, m\}}{\sup} \big\| \mathcal{W}_p\big(\bar{\mu}^N_{t_k},\nu^N_{t_k}\big) \big\|_p + \underset{k \in \{1,\ldots, m\}}{\sup} \bigg\| \mathcal{W}_p(\bar{\mu}_{t_k},\nu^N_{t_k})\bigg\|_p \notag \\     
    & \leq \frac{1}{N}\sum_{i=1}^N \underset{k \in \{1,\ldots, m\}}{\sup} \big\|\bar{X}^{i,N}_{t_k}-\bar{Y}^i_{t_k}\big\|_p + \underset{0\leq t\leq t_m}{\sup} \big\| \mathcal{W}_p\big(\bar{\mu}_{t},\nu^N_{t}\big) \big\|_p \notag  \\
    & \leq \bigg\| \underset{k \in \{1,\ldots, m\}}{\sup}\big|\bar{X}^{1,N}_{t_k}-\bar{Y}^1_{t_k}\big|\bigg\|_p + \underset{0\leq t\leq t_m}{\sup} \big\| \mathcal{W}_p\big(\bar{\mu}_{t},\nu^N_{t}\big) \big\|_p \notag  \\
    &\leq \bigg\| \underset{0\leq t\leq t_m}{\sup}\big|\bar{X}^{1,N}_{t}-\bar{Y}^1_{t}\big|\bigg\|_p + \underset{0\leq t\leq t_m}{\sup} \big\| \mathcal{W}_p\big(\bar{\mu}_{t},\nu^N_{t}\big) \big\|_p, \label{eq:bound wasserstein euler scheme empirical measure to conditional law}
\end{align}
\noindent
where the first inequality follows from the triangle inequality of the Wasserstein distance and the Minkowski inequality, the second inequality comes from the fact that the measure $\frac{1}{N}\sum_{i=1}^{N}\delta_{(x_i, y_i)}$ is a coupling of the measure $\frac{1}{N}\sum_{i=1}^{N}\delta_{x_i}$ and $\frac{1}{N}\sum_{i=1}^{N}\delta_{y_i}$, and the third inequality is due to the fact that $\big(\bar{X}^{i,N},\bar{Y}^i)_{i\in\{1\,,\,\ldots,\,N\}}$ are identically distributed random variables.
Remark that $(\bar{Y}^i_t)_{t\in[0,T]}$ are conditionally independent copies of the continuous extension $(\bar{X}_t)_{t\in[0,T]}$ of the Euler scheme defined by \eqref{eq:continuous euler scheme MKV}.
For the first term of \eqref{eq:bound wasserstein euler scheme empirical measure to conditional law}, we apply the Lipschitz continuity of the coefficients $b$, $\sigma$, $\sigma^0$ and the Burkölder-Davis-Gundy inequality (Lemma \ref{lem:BDG}) to get
\begin{align*}
    \bigg\|\underset{0\leq s\leq t}{\sup}& \big|\bar{X}^{1,N}_s-\bar{Y}^1_s\big|\bigg\|_p \leq \bigg\|\underset{0\leq u \leq t}{\sup} \int_0^u \Big(b\big(\underline{s},\bar{Y}^{1}_{\underline{s}},\bar{\mu}_{\underline{s}}\big) - b(\underline{s},\bar{X}^{1,N}_{\underline{s}}, \bar{\mu}^N_{\underline{s}})\Big)ds\bigg\|_p\\ 
    &\qquad + \bigg\|\underset{0\leq u \leq t}{\sup} \int_0^u \Big(\sigma(\underline{s},\bar{Y}^1_{\underline{s}},\bar{\mu}_{\underline{s}})- \sigma(\underline{s},\bar{X}^{1,N}_{\underline{s}}, \bar{\mu}^N_{\underline{s}})\Big)dW_s\bigg\|_p\\ 
    &\qquad + \bigg\| \underset{0\leq u \leq t}{\sup} \int_0^u\Big(\sigma^0(\underline{s},\bar{Y}^1_{\underline{s}},\bar{\mu}_{\underline{s}})- \sigma^0(\underline{s},\bar{X}^{1,N}_{\underline{s}}, \bar{\mu}^N_{\underline{s}})\Big)dW^0_s\bigg\|_p \\
    &\leq L \int_0^t \Big(\big\|\bar{X}^{1,N}_{\underline{u}} - \bar{Y}^1_{\underline{u}}\big\|_p + \big\|\mathcal{W}_p\big(\bar{\mu}^N_{\underline{u}}, \bar{\mu}_{\underline{u}}\big)\big\|_p\Big)du + \bigg[C_{d,p}L\int_0^t \Big(\big\|\bar{X}^{1,N}_{\underline{u}} - \bar{Y}^1_{\underline{u}}\big\|_p^2  + \big\|\mathcal{W}_p\big(\bar{\mu}^N_{\underline{u}}, \bar{\mu}_{\underline{u}}\big)\big\|_p^2\Big)du \bigg]^\frac{1}{2},
\end{align*}
\noindent
where the positive constant $C_{d,p}$ comes from the Burkölder-Davis-Gundy inequality, see Lemma \ref{lem:BDG}. It follows from Inequality \eqref{eq:bound wasserstein euler scheme empirical measure to conditional law}
\begin{equation}\label{eq: pre Gronwall inequality}
    \bigg\|\underset{0\leq s\leq t}{\sup} \big|\bar{X}^{1,N}_s - \bar{Y}^1_s\big|\bigg\|_p \leq 2L\int_0^t\bigg\|\underset{0\leq r\leq u}{\sup}\big|\bar{X}^{1,N}_r-\bar{Y}^1_r\big|\bigg\|_pdu + C_{d,p,L}\bigg(\int_0^t\bigg\|\underset{0\leq r\leq u}{\sup}\big|\bar{X}^{1,N}_r-\bar{Y}^1_r\big|\bigg\|_p^2du\bigg)^\frac{1}{2} + g(t),
\end{equation}
where 
\[ t\in[0,T] \longmapsto g(t) = L\int_0^t \underset{0\leq r\leq u}{\sup}\big\|\mathcal{W}_p\big(\bar{\mu}_r, \nu^N_r\big) \big\|_pdu + C_{d,p,L} \bigg(\int_0^t \underset{0\leq r\leq u}{\sup} \big\|\mathcal{W}_p\big(\bar{\mu}_r, \nu^N_r\big)\big\|_p^2du\bigg)^\frac{1}{2}. \]
We also define on $[0,T]$ the mapping 
\[ t\longmapsto f(t) = \bigg\|\underset{0\leq r\leq t}{\sup}\big|\bar{X}^{1,N}_r-\bar{Y}^1_r\big|\bigg\|_p. \]
It follows from Lemma \ref{lem:Lp euler scheme} that $(\bar{Y}^1_t)_{t\in[0,T]} \in \mathbb{L}^p([0,T]\times\Omega)$. By a similar reasoning, we can deduce that $(\bar{X}^{1,N}_t)_{t\in[0,T]} \in \mathbb{L}^p([0,T]\times\Omega)$. Moreover $(\bar{\mu}_t)_{t\in[0,T]}$ and $(\nu^N_t)_{t\in[0,T]}$ are random variables on $\Omega^0$ taking values in $\mathcal{C}([0,T],\mathcal{P}_p(\R^d))$. Hence the mappings $f$ and $g$ are continuous, non-decreasing and non-negative on $[0,T]$. We deduce from Inequality \eqref{eq: pre Gronwall inequality} and Lemma \ref{lem:a la Gronwall} that
\begin{equation} \label{eq:bound sup particle by wasserstein}
    \bigg\|\underset{0\leq s\leq t}{\sup} \big|\bar{X}^{1,N}_s-\bar{Y}^1_s\big|\bigg\|_p \leq 2e^{(2L + C^2_{d,p,L})t}\, g(t).
\end{equation}
Merging Inequalities \eqref{eq:bound wasserstein euler scheme empirical measure to conditional law} and \eqref{eq:bound sup particle by wasserstein}, we apply Lemma \ref{lem:a la Gronwall} once more to derive the following essential inequality
\begin{equation} \label{eq: bound sup wasserstein empirical}
    \underset{m \in \{1,\ldots, M\}}{\sup}\big\| \mathcal{W}_p(\bar{\mu}^N_{t_m},\bar{\mu}_{t_m})\big\|_p \leq C_{d,p,L,T}\, \underset{0\leq t\leq T}{\sup} \big\|\mathcal{W}_p\big(\bar{\mu}_{t},\nu^N_{t}\big) \big\|_p,
\end{equation}
where $C_{d,p,L,T}$ is a positive constant. Due to Lemma \ref{lem:MKV sol cond iid}, the $\bar{Y}^i$'s are identically distributed and independent conditionally to $W^0$ and in $\mathbb{L}^p$, by Lemma \ref{lem:Wp convergence empirical measure}, we deduce that for every $t\in[0,T]$, for every $\omega_0\in \Omega^0$, 
\[ \E^1\bigg[\mathcal{W}_p^p\big(\bar{\mu}_t(\omega_0), \nu^N_t(\omega_0)\big) \bigg] \xrightarrow[N\rightarrow +\infty]{}0. \]
Moreover we have the following control by the Young's inequality
\begin{equation} \label{eq:control cond expectation}
    \begin{aligned} 
        &\underset{0\leq t\leq T}{\sup} \E^1\big[\mathcal{W}^p_p\big(\bar{\mu}_t, \nu^N_t\big)\big] \leq 2^{p-1} \underset{0\leq t\leq T}{\sup} \E^1\big[\mathcal{W}^p_p\big(\bar{\mu}_{t}, \delta_0\big) + \mathcal{W}^p_p\big(\delta_0, \nu^N_t\big)\big] \\
        &\leq 2^{p-1}\frac{1}{N}\sum_{i=1}^N \underset{0\leq t\leq T}{\sup} \E^1\big[|\bar{Y}^1_t|^p\big] + 2^{p-1} \underset{0\leq t\leq T}{\sup} \E^1\big[|\bar{Y}^1_t|^p\big] \\
        &= 2^p\underset{0\leq t\leq T}{\sup}\E^1\big[|\bar{Y}^1_t|^p\big].
    \end{aligned}
\end{equation}
We conclude by using the $L^p$ version of the dominated convergence theorem and the Inequality \eqref{eq: bound sup wasserstein empirical} to deduce that
\[ \underset{m \in \{1,\ldots, M\}}{\sup}  \big\| \mathcal{W}_p(\bar{\mu}^N_{t_m},\bar{\mu}_{t_m})\big\|_p \leq C_{d,p,L,T}\, \underset{0\leq t\leq T}{\sup} \big\| \mathcal{W}_p\big(\bar{\mu}_{t},\nu^N_{t}\big) \big\|_p \xrightarrow[N\rightarrow +\infty]{}0. \]
\medskip \\
\textit{(ii)} Additionally, we assume that $\|X_0\|_{p+\varepsilon} <\infty$ for $\varepsilon >0$. Hence, it follows from Lemma \ref{lem:Lp euler scheme} that $\bar{Y}^i \in \mathbb{L}^{p+\varepsilon}$ for every $i \in \{1, \ldots, N\}$. Moreover the $\bar{Y}^i$'s are i.i.d. given the path of $W^0$ according to Lemma \ref{lem:MKV sol cond iid}. Applying Lemma \ref{lem:FournierGuillin}, we establish the following rate of convergence for every $\omega_0\in \Omega^0$,
\begin{equation*}
    \underset{0\leq t\leq T}{\sup} \E^1\Big[ \mathcal{W}_p^p\big(\bar{\mu}_t(\omega_0), \nu^N_t(\omega_0)\big) \Big] \leq \tilde{C} \left\{ 
    \begin{array}{ll}
        N^{-\frac{1}{2}} + N^{-\frac{\varepsilon}{(p+\varepsilon)}} & \text{ if } p>\frac{d}{2} \text{ and } \varepsilon\neq p, \\
        N^{-\frac{1}{2}}\big(\log(1+N)\big) + N^{-\frac{\varepsilon}{(p+\varepsilon)}} & \text{ if } p=\frac{d}{2} \text{ and } \varepsilon\neq p,\\
        N^{-\frac{1}{d}}+N^{-\frac{\varepsilon}{(p+\varepsilon)}} & \text{ if } p\in(0,\frac{d}{2}) \text{ and } p+\varepsilon\neq \frac{d}{d-p},
    \end{array}\right.
\end{equation*}
where $\tilde{C}$ is a constant independent of $N$. By Inequality \eqref{eq:control cond expectation}, we can apply the $L^p$ version of the dominated convergence theorem to deduce a similar convergence rate for $\underset{0\leq t\leq T}{\sup} \big\| \mathcal{W}_p^p(\bar{\mu}_t, \nu^N_t) \big\|_p$. From the Inequality \eqref{eq: bound sup wasserstein empirical} and Remark \ref{rq:cv rate PoC}, we can conclude that
\begin{align*}
    \underset{m \in \{1,\ldots, M\}}{\max} \big\| \mathcal{W}_p(\bar{\mu}^N_{t_m}, \bar{\mu}_{t_m}) \big\|_p &\leq \sup_{t\in[0,T]}\, \big\| \mathcal{W}_p(\bar{\mu}^N_{t}, \bar{\mu}_{t}) \big\|_p \\
    &\leq C \left\{ 
    \begin{array}{ll}
        N^{-\frac{1}{2p}} + N^{-\frac{\varepsilon}{p(p+\varepsilon)}} & \text{ if } p>\frac{d}{2} \text{ and } \varepsilon\neq p, \\
        N^{-\frac{1}{2p}}\big(\log(1+N)\big)^\frac{1}{p} + N^{-\frac{\varepsilon}{p(p+\varepsilon)}} & \text{ if } p=\frac{d}{2} \text{ and } \varepsilon\neq p,\\
        N^{-\frac{1}{d}}+N^{-\frac{\varepsilon}{p(p+\varepsilon)}} & \text{ if } p\in(0,\frac{d}{2}) \text{ and } p+\varepsilon\neq \frac{d}{d-p},
    \end{array}\right.
\end{align*}
where $C$ is a positive constant that depends on the parameters $d$, $p$, $L$, $T$ and the data $(\|X_0\|_{p+\varepsilon}, b, \sigma, \sigma^0)$.
\end{proof}

\medskip

\begin{proof}[Proof of Theorem \ref{thm:cv rate scheme}]
Recall that the processes $(\bar{Y}^i_t)_{t\in[0,T]}, 1\leq i\leq N$ are defined by the continuous Euler scheme \eqref{eq:continuous euler solution scheme}. Then we use the Minkowski inequality to get that
\begin{equation} \label{eq:norme sup scheme divided}
    \bigg\|\sup_{t\in[0,T]}\big|\bar{X}^{1,N}_t - X_t\big|\bigg\|_p \leq \bigg\|\sup_{t\in[0,T]}\big|\bar{X}^{1,N}_t - \bar{Y}^{1}_t\big|\bigg\|_p + \bigg\|\sup_{t\in[0,T]}\big|\bar{Y}^{1}_t - X_t\big|\bigg\|_p.
\end{equation}
\noindent
By the inequality \eqref{eq:bound sup particle by wasserstein}, we have
\begin{equation*}
    \bigg\| \sup_{t\in [0,T]} \big|\bar{X}^{1,N}_{t} - \bar{Y}^{1}_{t}\big| \bigg\|_p \leq C_{d,p,L,T}\, \sup_{t\in[0,T]}\, \big\| \mathcal{W}_p(\bar{\mu}^N_{t}, \bar{\mu}_{t}) \big\|_p.
\end{equation*}
From Theorem \ref{thm:particle method cv rate}, we derive the following rate of convergence
\begin{equation}\label{eq:cv rate particle method}
    \sup_{t\in[0,T]}\, \big\| \mathcal{W}_p(\bar{\mu}^N_{t}, \bar{\mu}_{t}) \big\|_p \leq C_{d,p,L,T}\,\mathcal{E}_N,
\end{equation}
where 
\begin{equation*}
    \mathcal{E}_N = \left\{
    \begin{array}{ll}
        N^{-\frac{1}{2p}} + N^{-\frac{\varepsilon}{p(p+\varepsilon)}} & \text{ if } p>\frac{d}{2} \text{ and } \varepsilon\neq p, \\
        N^{-\frac{1}{2p}}\big(\log(1+N)\big)^\frac{1}{p} + N^{-\frac{\varepsilon}{p(p+\varepsilon)}} & \text{ if } p=\frac{d}{2} \text{ and } \varepsilon\neq p,\\
        N^{-\frac{1}{d}}+N^{-\frac{\varepsilon}{p(p+\varepsilon)}} & \text{ if } p\in(0,\frac{d}{2}) \text{ and } p+\varepsilon\neq \frac{d}{d-p}.
    \end{array}\right.
\end{equation*}
Moreover, it follows from Proposition \ref{prop:euler scheme cv} that
\begin{equation}\label{eq:cv rate euler scheme}
    \bigg\|\sup_{t\in[0,T]}\big|X_t - \bar{Y}^1_t\big|\bigg\|_p \leq C_{d,p,L,T}\,h^{\frac{1}{2}\wedge\rho},
\end{equation}
since $\bar{X}^{1,N}_t$ shares the same Brownian motions as $X_t$, so does $\bar{Y}^1_t$. We plug Inequalities \eqref{eq:cv rate particle method} and \eqref{eq:cv rate euler scheme} into Inequality \eqref{eq:norme sup scheme divided} to deduce the rate of convergence in Inequality \eqref{eq:total cv rate}.
\end{proof}

\medskip
\section{Numerical simulations} \label{section : simulation examples}

This section presents two simulation examples, the first one is a conditional Ornstein-Uhlenbeck process with common noise and the second one is an application from \cite{ren2024risk} representing an Interbank market model. The simulation codes are available via Google Colab (\url{https://bit.ly/49xgGHG}, \url{https://bit.ly/3D4zwK3}).

\subsection{Conditional Ornstein-Uhlenbeck process with common noise}

In this section, we present the following simulation example in $\R^d$
\begin{equation} \label{eq:simu example}
 dX_t =-\big(X_t - \E^1(X_t)\big)dt + \sigma dW_t + \sigma^0 dW_t^0, \quad X_0=x_0\in\R^d
\end{equation}
where $\sigma$ and $\sigma^0$ are two positive scalar. This example generalizes, incorporating common noise, the one analyzed in detail in \cite[Section 3.1]{Lacker}. It is obvious that Equation \eqref{eq:simu example} admits a unique solution $X=(X_t)_{t\in[0,T]}$, as the drift coefficient $b$ is linear in $x$ and in $\mu$, and the diffusion coefficients $\sigma$ and $\sigma^0$ are constants. A straightforward computation yields the following closed-form expression for the unique solution
\begin{equation}\label{eq:explicit solution example}
    X_t = e^{-t} x_0 + \sigma \int_0^t e^{-(t-s)} dW_s + \sigma^0 W^0_t.
\end{equation}
We set $x_0 = 0$, $T=1$ and $\sigma = \sigma^0 = \sqrt{0.2}$. In a first time, we try to compute the simulation error by computing the $\mathbb{L}^2$-error between the solution process $(X_t)_{t\in[0,T]}$ and its approximation. \smallskip \\

The numerical study of this equation is divided into two parts. In the first part, we consider the case where the dimension  $d=2$ and analyze the $L^2$-error $\{\mathbb{E} [\sup_{t\in[0,T]}|X_t-\bar{X}_t^{1,N}|^2]\}^{1/2}$, where $(X_t)_{t\in[0,T]}$ is the true solution defined by \eqref{eq:explicit solution example} and $(\bar{X}_t^{1,N})_{t\in[0,T]}$ is the first particle path defined by \eqref{eq:particle method continuous scheme} sharing the same Brownian motion of $(X_t)_{t\in[0,T]}$. This $L^2$-error is approximated by 
\begin{equation}\label{eq:approxi-l2-error} \varepsilon_N^2 = \frac{1}{30} \sum_{j=1}^{30} \frac{1}{N} \sum_{i=1}^N \max_{1\leq m\leq M} \big|\bar{X}^{i,N,j}_{t_m} - X^{i,j}_{t_m}\big|^2,
\end{equation}
using the time discretization number $M=100$. Here, in \eqref{eq:approxi-l2-error}, for every fixed $i\in \{1, ..., N\}$, $j\in \{1, ..., 30\}$,  the random variables $\bar{X}^{i,N,j}_{t_m}$ is simulated using the particle method \eqref{eq:euler particle system}, and $\bar{X}^{i,j}_{t_m}$ is computed using the explicit solution  \eqref{eq:explicit solution example} both sharing the same Brownian motion.  Note that the mean error over $j\in \{1, .., 30\}$ is used for the Monte-Carlo estimation of the expectation. In Figure \ref{fig:MeanSquaredError}, we display the log-log error of $\varepsilon_N$ as a mapping of the number of particles $N \in \{2^{6}, \, \ldots,\,2^{16}\}$.

\noindent



The second part focuses on the one-dimensional case. Remark that this example has an explicit formula for the density of $(\mu_t)_{t\in[0,T]} = (\mathcal{L}^1(X_t))_{t\in[0,T]}$, given by  
\begin{equation*}
\frac{\mu_t(x)}{dx} = \frac{1}{\sqrt{2\pi\sigma^2(1-e^{-t})}} \exp\bigg(-\frac{|x - X_0 -\sigma_0W^0_t|^2}{2\sigma^2(1-e^{-t})}\bigg).
\end{equation*}
Hence, we can compute an approximated density $(\hat{\mu}^{N, h, \eta}_t)_{t\in[0,T]}$ by using the kernel method as presented in \cite{DensitySimuMKV} for a standard McKean-Vlasov equation. The estimator $\hat{\mu}^{N, h, \eta}_T$ of the density $\mu_T$ at the final time $T=1$ is defined, for every $x\in\R$, by 
\begin{equation*}
\widehat{\mu}_T^{N, h, \eta}(x):=N^{-1} \sum_{i=1}^N \eta^{-1} K\left(\eta^{-1}\left(x-\bar{X}_T^{i, N}\right)\right), 
\end{equation*}
where $K(\cdot)$ is the kernel function and $\eta$ is the bandwidth. We will apply a Gaussian-based kernel of order $l=5$, $K(x)= \frac{1}{8}(15 - 10x^2 + x^4)\phi(x)$ where $\phi(x) = (2\pi)^{-1/2}\exp{-\frac{x^2}{2}}$, with the bandwidth $\eta=N^{-1/(2(l+1)+1)}$ chosen according to Corollary 2.11 of \cite{DensitySimuMKV}, since $\mu_T \in \mathcal{C}^\infty_b(\R)$. The estimation error is computed by 
\begin{equation}\label{eq:approxi-density-error}
    \mathcal{E}_N^2=\frac{1}{30} \sum_{j=1}^{30} \max_{x \in \mathcal{D}}\left|\left(\widehat{\mu}_T^{N, h, \eta}\right)_j(x)-\mu_T(x)\right|^2,
\end{equation}
where the domain $\mathcal{D}$ is a uniform grid on $[-3,3]$ chosen according to the trajectory of the simulations $(\bar{X}^{i,N})_{1\leq i\leq N}$. Figure \ref{fig:density_kernels_error} shows the log-log error of this density estimation.

\begin{figure}[H]
\centering
\begin{minipage}[t]{0.5\textwidth}
   \centering
    \includegraphics[width=1\linewidth]{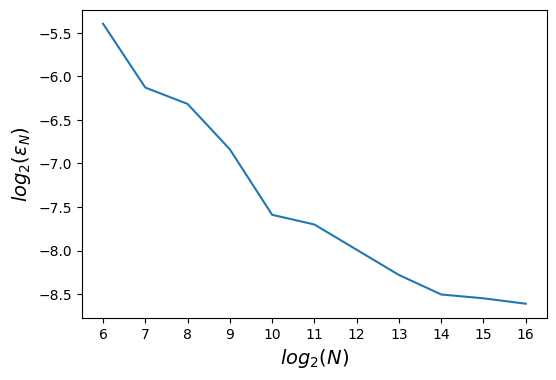}
    \caption{Log-log error \eqref{eq:approxi-l2-error} between $(X_t)_{t\in[0,T]}$ and $(\bar{X}_{t_m})_{0\leq m\leq M}$ (slope = -0.35)}
    \label{fig:MeanSquaredError}
\end{minipage}
\hspace{-0.2cm}\begin{minipage}[t]{0.5\textwidth}
   \centering
    \includegraphics[width=1\linewidth]{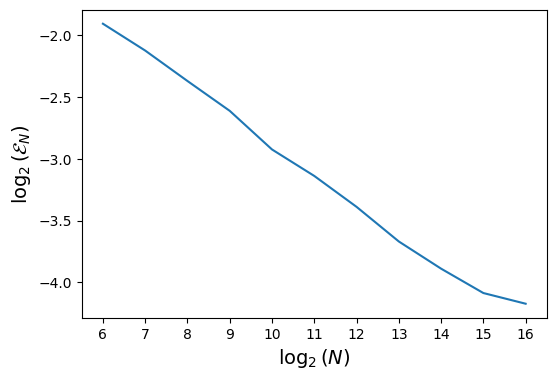}
    \caption{Log-log error \eqref{eq:approxi-density-error} between $\hat{\mu}^{N,h,\eta}_T$ and $\mu_T$ (slope = -0.4)}
    \label{fig:density_kernels_error}
\end{minipage}
\end{figure}

\subsection{Interbank market model}

We consider an application from \cite{ren2024risk} where they study risk-sensitive mean field games (MFG) with common noise. It is an infinite population model where the log-reserve $(X^{i,N}_t)_{t\in[0,T]}$ of the bank $i\in I$ at time $t\in[0,T]$ satisfies the following dynamics 
\begin{equation} \label{eq:MKV_interbank}
dX^{i,N}_t = \Big(a(\bar{X}_t - X^{i,N}_t) + u^i_t +b(t) \Big)dt + \sigma \sqrt{1-\rho^2}dW^i_t + \sigma\rho dW^0_t,
\end{equation}
with $\bar{X}_t = \lim_{N\rightarrow \infty} \frac{1}{N}\sum_{i=1}^N X^{i,N}_t$ represents the limiting market state which is following the dynamic
\begin{equation}\label{eq:interbank_ex-eq2}
d\bar{X}_t = (\bar{u}_t+b(t))dt + \sigma\rho dW^0_t,
\end{equation}
where $\bar{u}_t = \lim_{N\rightarrow \infty} \frac{1}{N}\sum_{i=1}^N u^i_t $. The transaction $u^i_t \in \mathbb{R}$ represents the money that the bank lends to or borrows from the central bank during the market activity at each time $t$, the market shock is simulated by $W^0_t \in \mathbb{R}$ which is independent of the shock received by the bank $W^i_t \in \mathbb{R}$. 

\noindent
The parameter $a \in \mathbb{R}$ represents the mean reversion rate of the bank's reserve towards the market state. The bank's liquidity before market activity at each time $t$ is denoted by $b(t)$. The volatility of the log-reserve of the bank with respect to its own local shock (underlying
uncertainty source) is given by $\sigma\rho \in \mathbb{R}$. Meanwhile, the volatility of the log-reserve with respect to the global shock that affects the market (i.e. the macroeconomic factors), is characterized by $\sigma\sqrt{1-\rho^2}\in\mathbb{R} $. As can be seen from above equations, an instantaneous coefficient $0\leq \rho \leq 1$ is a common multiplier factor for the shock delivered by the bank itself and by the environment.

The error analysis of this Interbank market model will proceed as follows. First, we fix $b(t) = 1$ for every $t\in[0,T]$, set $a=10$ and choose $\rho \in (0,1)$ to ensure the presence of both  non-zero idiosyncratic noises and common noise. Next, we compute the optimal transaction rate $(u^{i,\star}_t)_{t\in[0,T]}$ with respect to the cost function defined by \cite[Equation (134)]{ren2024risk}. It is important to note that the explicit formula for $(u^{i,\star}_t)_{t\in[0,T]}$ depends on parameters $(\Pi_t, \Lambda_t, \Upsilon_t, \Delta_t, \Gamma_t, \Psi_t)_{t\in[0,T]}$,  as specified in \cite[Equations (139-144)]{ren2024risk} and these parameters are determined by a system of ODEs, which can be numerically solved  using the Python library \texttt{solve\_ivp}  from \texttt{scipy.integrate}.

After computing the optimal transaction rate $(u^{i,\star}_t)_{t\in[0,T]}$, the value of $\bar{u}^{\star}_t$ can be directly deduced from  \cite[Equation (145)]{ren2024risk} and plugged into the dynamics of $(\bar{X}_t)_{t\in[0,T]}$ defined by \eqref{eq:interbank_ex-eq2}. Consequently, it can be numerically solved using a standard Euler scheme for diffusion processes (see e.g. \cite[Section 7.1]{NumericalProbability}). On the other hand, the dynamics described by \eqref{eq:MKV_interbank}, incorporating the optimal transaction rates $(u^{i,\star}_t)_{t\in[0,T]}$, can be computed using the particle method presented in this paper. 

The simulation error is therefore evaluated as the difference between $\frac{1}{N}\sum_{i=1}^N X^{i,N}_t$, obtained from the following dynamics 
\begin{equation}
dX^{i,N}_t = \left\{10\left(\frac{1}{N}\sum_{i=1}^{N}X_t^{i,N}-X_t^{i,N}\right) + u^{i,\star}_t + 1 \right\} dt + \sigma \sqrt{1 - \rho^2} dW_t^i + \sigma \rho dW_t^0, \smallskip \label{eq:interbank-eq3} 
\end{equation}
and $\bar{X}_t$ computed by \eqref{eq:interbank_ex-eq2}, and we perform 30 independent experiments for Monte-Carlo approximation of the expectation. More specifically, we set the time discretization number $M=100$ for both  \eqref{eq:interbank_ex-eq2} and \eqref{eq:interbank-eq3}, and compute the simulation error using the following formula:
\begin{equation}\label{eq:interbank-error}
    \mathcal{E}_N^2 = \frac{1}{30}\sum_{j=1}^{30} \max_{1\leq m \leq M} \left\| \frac{1}{N}\sum_{i=1}^N X^{i,N,j}_{t_m} - \bar{X}^{j}_{t_m} \right\|^2.
\end{equation}
where $j$ in the superscript of  $\bar{X}^{i,N,j}_{t_m}$ and $\bar{X}^{j}_{t_m}$ denotes the index of the independent experiment, and for each fixed $j\in \{1, ..., 30\}$, $\bar{X}^{i,N,j}_{t_m}, 1\leq i\leq N$ is computed using the particle method \eqref{eq:euler particle system} based on the dynamics in \eqref{eq:interbank-eq3}, while $\bar{X}^{j}_{t_m}$ is computed using a standard Euler scheme for the diffusion process \eqref{eq:interbank_ex-eq2}. Figure \ref{fig:MFG_MeanSquaredError} illustrates the log-log error of $ \mathcal{E}_N^2$ defined by \eqref{eq:interbank-error} and Figure \ref{fig:MFG_simulation_examples} displays 10 simulated paths of $(X^{i,N}_t)_{t\in[0,T]}$ under the above setting. 

\begin{figure}[H]
\centering
\begin{minipage}[t]{0.5\textwidth}
   \centering
    \includegraphics[width=0.95\linewidth]{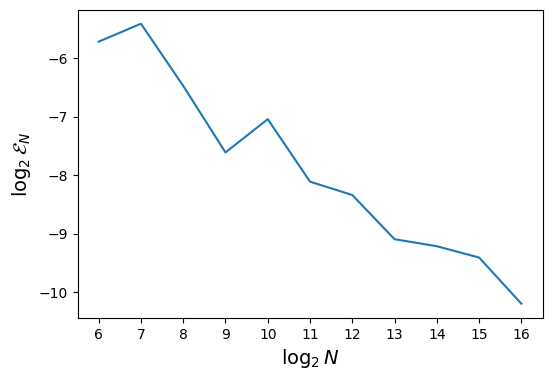}
    \caption{Log-log error of $\mathcal{E}_N$ defined by \eqref{eq:interbank-error} (slope = -0.4)}
    \label{fig:MFG_MeanSquaredError}
\end{minipage}
\hspace{-0.2cm}\begin{minipage}[t]{0.5\textwidth}
   \centering
    \includegraphics[width=0.95\linewidth]{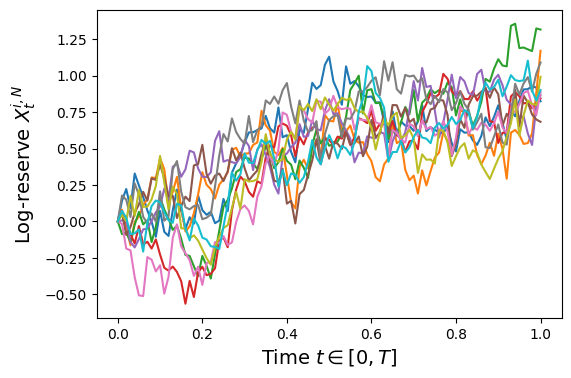}
    \caption{10 paths of particles $(X^{i,N}_t)_{t\in[0,T]}$}
    \label{fig:MFG_simulation_examples}
\end{minipage}
\end{figure}

\noindent


\appendix
\section{Appendix}\label{section:appendix}

\begin{proof}[Proof of Lemma \ref{lem:cond expec inequality}]
For a fixed $\omega^0 \in \Omega^0$, we have
\begin{align*}
    \mathcal{W}_p^p\big(\mathcal{L}^1(Y^1)(\omega^0), \,\mathcal{L}^1(Y^2)(\omega^0)\big) &= \mathcal{W}_p^p\big(\mathcal{L}(Y^1(\omega^0, \cdot)),\, \mathcal{L}(Y(\omega^0,\, \cdot))\big) \\
    & \leq \E^1 \big[ \big| Y_1(\omega^0,\, \cdot) - Y_2(\omega^0,\, \cdot)\big|^p \big]
\end{align*}
This inequality is true for every $\omega^0 \in \Omega^0$, hence we have
\begin{equation*}
    \E\big[\mathcal{W}_p^p(\mathcal{L}^1(Y_1), \mathcal{L}^1(Y_2))\big] \leq \E\big[|Y_1-Y_2\|^p \big]
\end{equation*}
Then we have 
\[ \big\|\mathcal{W}_p(\mathcal{L}^1(Y_1), \mathcal{L}^1(Y_2))\big\|_p \leq \|Y_1-Y_2\|_p. \hfill \qedhere \]
\end{proof}
\medskip

\noindent
We present the proofs of Lemma \ref{lem:Lp euler scheme} and Lemma \ref{lem:holder continuity solution} required by the proof of Proposition \ref{prop:euler scheme cv}.

\begin{proof}[Proof of Lemma \ref{lem:Lp euler scheme}]
We reason by forward induction. By Minkowski's inequality, we have
\begin{align*}
    \big\|\bar{X}_{t_1}\big\|_p &= \big\|\bar{X}_0 + h\cdot b(0,\bar{X}_0,\bar{\mu}_0)t_1+ \sigma(0,\bar{X}_0,\bar{\mu}_0)W_{t_1} + \sigma^0(0,\bar{X}_0,\bar{\mu}_0)W^0_{t_1}\big\|_p \\
    &\leq \big\|\bar{X}_0\big\|_p + h\big\|b(0,\bar{X}_0,\bar{\mu}_0)\big\|_p + \sqrt{h}\big\|\sigma(0,\bar{X}_0,\bar{\mu}_0)Z_1\big\|_p + \sqrt{h}\big\|\sigma^0(0,\bar{X}_0,\bar{\mu}_0)Z^0_1\big\|_p 
\end{align*}
Since the solution $(X_t)_{t\in[0,T]}$ and the continuous scheme $(\bar{X}_t)_{t\in[0,T]}$ have the same initial condition $X_0$, we have, $\mathbb{P}$-almost surely  
\[\bar{X}_0=X_0\in L^p(\Omega) \text{ and } \bar{\mu}_0=\mathcal{L}(X_0). \] 
For every $0\leq m\leq M-1$, $Z_{m+1}$ and $Z^0_{m+1}$ are independent of $(\bar{X}_{t_m},\, \ldots, \bar{X}_{t_1},X_0)$ and identically distributed. Moreover, since $b$, $\sigma$ and $\sigma^0$ have a linear growth, we have by Minkowski's inequality
\begin{align*}
    \big\|\bar{X}_{t_1}\big\|_p &\leq \|X_0\|_p + C_{p,L,T,b,\sigma,\sigma^0}\big(1+\|X_0\|_p + \big\|\mathcal{W}_p(\mu_0,\delta_0)\big\|_p\big)\big(h+\sqrt{h}\|Z_1\|_p+\sqrt{h}\|Z^0_1\|_p \big) \\
    &\leq \|X_0\|_p + C_{p,L,T,b,\sigma,\sigma^0}\big(1+2\|X_0\|_p\big)(h+\sqrt{h}).
\end{align*}
For $1\leq m \leq M-1$, we repeat the same reasoning to show that 
\begin{align*}
     \big\|\bar{X}_{t_{m+1}}\big\|_p &\leq \|\bar{X}_{t_m}\|_p + C_{p,L,T,b,\sigma,\sigma^0}\big(1+2\|\bar{X}_{t_m}\|_p\big)(h+ C_p\sqrt{h}) \\
     &\leq C_{p,L,T,b,\sigma,\sigma^0,h}^0 \|\bar{X}_{t_m}\|_p + C_{p,L,T,b,\sigma,\sigma^0,h}^1.
\end{align*}
By forward induction, we have for every $0\leq m \leq M$,  
\begin{equation} \label{eq:bound Lp euler scheme step}
    \big\|\bar{X}_{t_m}\big\|_p \leq C\big(1+\|X_0\|_p\big),
\end{equation}
where $C$ is a positive constant depending on the parameters $p$, $L$, $T$, $h$, $M$ and the coefficients $b$, $\sigma$ and $\sigma^0$.
For every $t\in[0,T]$, we recall the notation $\underline{t}=t_m$ where $m\in\{0,\ldots,M\}$ such that $t_m \leq t <t_{m+1}$. By definition \eqref{eq:continuous euler scheme MKV} of the continuous Euler scheme $(\bar{X}_t)_{t\in[0,T]}$, it is $\mathbb{F}$-adapted and right continuous so is progressively measurable. Moreover, we have
\begin{align}
    &\bigg\|\underset{0\leq t\leq T}{\sup}\big|\bar{X}_t\big|\bigg\|_p = \bigg\|\underset{0\leq t\leq T}{\sup} \Big|X_0 + 
    \int_0^t b\big(\underline{s},\bar{X}_{\underline{s}}, \bar{\mu}_{\underline{s}}\big)ds + \int_0^t \sigma\big(\underline{s},\bar{X}_{\underline{s}}, \bar{\mu}_{\underline{s}}\big)dW_s  + \int_0^t \sigma^0\big(\underline{s},\bar{X}_{\underline{s}}, \bar{\mu}_{\underline{s}}\big)dW^0_s \Big| \bigg\|_p \notag \\
    &\leq \|X_0\|_p + \bigg\|\underset{0\leq t\leq T}{\sup} \Big| \int_0^t b\big(\underline{s},\bar{X}_{\underline{s}}, \bar{\mu}_{\underline{s}}\big)ds \Big|\bigg\|_p + \bigg\|\underset{0\leq t\leq T}{\sup} \Big|\int_0^t \sigma\big(\underline{s},\bar{X}_{\underline{s}}, \bar{\mu}_{\underline{s}}\big)dW_s\Big|\bigg\|_p \notag \\
    &\hspace{1.6cm} +\bigg\|\underset{0\leq t\leq T}{\sup} \Big| \int_0^t \sigma^0\big(\underline{s},\bar{X}_{\underline{s}}, \bar{\mu}_{\underline{s}}\big)dW^0_s \Big| \bigg\|_p \notag \\ 
    & \leq \|X_0\|_p + \bigg\| \int_0^T \Big| b\big(\underline{s},\bar{X}_{\underline{s}}, \bar{\mu}_{\underline{s}}\big) \Big| ds \bigg\|_p + C_{d,p}\bigg\| \Big(\int_0^T \big|\sigma\big(\underline{s},\bar{X}_{\underline{s}}, \bar{\mu}_{\underline{s}}\big)\big|^2 ds\Big)^{\frac{1}{2}}\bigg\|_p \\ 
    & \quad + C_{d,p}\bigg\| \Big(\int_0^T \big|\sigma^0\big(\underline{s},\bar{X}_{\underline{s}}, \bar{\mu}_{\underline{s}}\big)\big|^2 ds\Big)^{\frac{1}{2}}\bigg\|_p \notag \\
    & \leq \|X_0\|_p + C_{L,T,b,\sigma,\sigma^0}\bigg[ \int_0^T \Big\| \big(1+|\bar{X}_{\underline{s}}| + \mathcal{W}_p(\bar{\mu}_{\underline{s}}, \delta_0)\big) \Big\|_p ds + \bigg\| \Big(\int_0^T \big|1+|\bar{X}_{\underline{s}}| + \mathcal{W}_p(\bar{\mu}_{\underline{s}},\delta_0)\big|^2ds\Big)^{\frac{1}{2}}\bigg\|_p \bigg] \label{eq:bound sup euler Lp}
\end{align}
where we use Minkowski's inequality for the first inequality, the second inequality follows from Lemma \ref{lem:BDG} and the linear growth of the coefficients. By applying Young's inequality on \eqref{eq:bound sup euler Lp}, we can infer that
\begin{align*}
    &\bigg\|\underset{0\leq t\leq T}{\sup}\big|\bar{X}_t\big|\bigg\|_p \\
    & \leq \|X_0\|_p + C_{L,T,d,p,b,\sigma,\sigma^0} \bigg[\int_0^T \big(1+2\|\bar{X}_{\underline{s}}\|_p\big) ds + \sqrt{T} + \bigg(\int_0^T \|\bar{X}_{\underline{s}}\|_p^2ds \bigg)^\frac{1}{2} + \bigg(\int_0^T \mathcal{W}_p^2(\bar{\mu}_{\underline{s}},\delta_0)ds \bigg)^\frac{1}{2}  \bigg] \\
    & \leq \|X_0\|_p + C_{L,T,d,p,b,\sigma,\sigma^0} \bigg[\int_0^T \big(1+2\|\bar{X}_{\underline{s}}\|_p\big) ds + \bigg(\int_0^T \|\bar{X}_{\underline{s}}\|_p^2ds \bigg)^\frac{1}{2} \bigg],
\end{align*}
from Lemma \ref{lem:cond expec inequality}. Then we deduce the following inequality to which we apply Lemma \ref{lem:a la Gronwall}
\begin{equation}
    \bigg\|\underset{0\leq t\leq T}{\sup}\big|\bar{X}_t\big|\bigg\|_p \leq \|X_0\|_p + C_{L,T,d,p,b,\sigma,\sigma^0} \bigg[\int_0^T \Big\|\sup_{0\leq u\leq s} |\bar{X}^M_u|\Big\|_p ds + \bigg(\int_0^T \Big\|\sup_{0\leq u\leq s} |\bar{X}^M_u|\Big\|_p^2ds \bigg)^\frac{1}{2} \bigg].
\end{equation}
Since $\bar{X}^M \in \mathbb{L}^p$, the application $t\mapsto\Big\|\sup_{0\leq u\leq t} \bar{X}^M_u\Big\|_p$ is continuous, non-negative and non-decreasing on $[0,T]$, the final estimate follows from Lemma \ref{lem:a la Gronwall}
\begin{equation*}
    \bigg\|\underset{0\leq t\leq T}{\sup}\big|\bar{X}_t\big|\bigg\|_p \leq \tilde{C}\big(1+\|X_0\|_p\big),
\end{equation*}
where $\tilde{C}$ is a positive constant depending on $p$, $L$, $T$, $h$, $M$ and the coefficients $b$, $\sigma$ and $\sigma^0$. Hence $\bar{X} \in\mathbb{L}^p([0,T]\times\Omega)$.
\end{proof}

\begin{proof}[Proof of Lemma \ref{lem:L1_cond_law}]
The first step of the proof is inspired by the proof of \cite[Proposition 2.9]{Carmona_Delarue2}. The first additional step is to combine the results of the Lemmas \ref{lem: 2.5 Carmona} and \ref{lem:Lp euler scheme}, that is $(\bar{X}_t)_{t\in[0,T]} \in \mathbb{L}^p([0,T] \times \Omega)$ and $\mathcal{L}^1(\bar{X}_t)_{t\in[0,T]}$ has continuous paths in $\mathcal{P}_p(\R^d)$ and is $\mathbb{F}^0$-adapted. The second one follows directly from the construction of the Euler scheme $(\bar{X}_t)_{t\in[0,T]}$. There exists a unique process that satisfies Equation \eqref{eq:continuous euler scheme MKV} for the Brownian motions $(W_t)_{t\in[0,T]}$ and $(W^0_t)_{t\in[0,T]}$. 
The second step is a consequence of Lemma \ref{lem: 2.5 Carmona} and Lemma \ref{lem:Lp euler scheme}.
\end{proof}

\begin{proof}[Proof of Lemma \ref{lem:holder continuity solution}]
Set $(\mu_t)_{t\in[0,T]} = \big(\mathcal{L}^1(X_t)\big)_{t\in[0,T]}$. Fix $s,t\in[0,T]$ such that $s\leq t$. Since $(X_t)_{t\in[0,T]}$ solves the McKean-Vlasov equation with common noise \eqref{eq:MKV}, we have
\begin{align*}
    &\big\|X_t-X_s\big\|_p = \bigg\|\int_s^t b(u,X_u,\mu_u)du + \int_s^t\sigma(u,X_u,\mu_u)dW_u + \int_s^t\sigma^0(u,X_u,\mu_u)dW^0_u\bigg\|_p \\
    &\leq \bigg\|\int_s^t b(u,X_u,\mu_u)du\bigg\|_p + \bigg\|\int_s^t\sigma(u,X_u,\mu_u)dW_u\bigg\|_p + \bigg\|\int_s^t\sigma^0(u,X_u,\mu_u)dW^0_u\bigg\|_p \\
    &\leq \int_s^t \big\|b(u,X_u,\mu_u)\big\|_pdu + C_{d,p} \bigg\|\Big(\int_s^t\sigma(u,X_u,\mu_u)^2du\Big)^\frac{1}{2}\bigg\|_p + C_{d,p} \bigg\|\Big(\int_s^t\sigma^0(u,X_u,\mu_u)^2du\Big)^\frac{1}{2}\bigg\|_p \\
    &=  \int_s^t \big\|b(u,X_u,\mu_u)\big\|_pdu + C_{d,p} \bigg\|\int_s^t\sigma(u,X_u,\mu_u)^2du\bigg\|_\frac{p}{2}^\frac{1}{2} + C_{d,p} \bigg\|\int_s^t\sigma^0(u,X_u,\mu_u)^2du\bigg\|_\frac{p}{2}^\frac{1}{2} \\
    &\leq \int_s^t \big\|b(u,X_u,\mu_u)\big\|_pdu + C_{d,p} \bigg(\int_s^t\big\|\sigma(u,X_u,\mu_u)^2\big\|_\frac{p}{2}du\bigg)^\frac{1}{2} + C_{d,p} \bigg(\int_s^t\big\|\sigma^0(u,X_u,\mu_u)^2\big\|_\frac{p}{2}du\bigg)^\frac{1}{2} \\
    & = \int_s^t \big\|b(u,X_u,\mu_u)\big\|_pdu + C_{d,p} \bigg(\int_s^t\big\|\sigma(u,X_u,\mu_u)\big\|_p^2du\bigg)^\frac{1}{2} + C_{d,p} \bigg(\int_s^t\big\|\sigma^0(u,X_u,\mu_u)\big\|_p^2du\bigg)^\frac{1}{2}.
\end{align*}
where we used the Minkowski's inequality at the first inequality, the second one follows from the general Minkowski's (Lemma \ref{lem:general minkowski}) and Burkölder-Davis-Gundy inequalities (Lemma \ref{lem:BDG}) and the general Minkowski's inequality provides the last one. We use the linear growth of the coefficients and Minkowski's and Young's inequalities to get that 
\begin{align*}
    &\big\|X_t-X_s\big\|_p \leq C_{L,T,b,\sigma,\sigma^0} \bigg(\int_s^t\big\|1+|X_u|+\mathcal{W}_p(\mu_u,\delta_0)\big\|_pdu + 2\bigg[ \int_s^t\big\|1+|X_u|+\mathcal{W}_p(\mu_u,\delta_0)\big\|_p^2du \bigg]^\frac{1}{2} \bigg)\\
    &\leq C_{L,T,b,\sigma,\sigma^0} \bigg(\int_s^t\big(1+\big\|X_u\big\|_p+\big\|\mathcal{W}_p(\mu_u,\delta_0)\big\|_p\big)du + 2\bigg[ \int_s^t\big(1+\big\|X_u\big\|_p+\big\|\mathcal{W}_p(\mu_u,\delta_0)\big\|_p\big)^2du \bigg]^\frac{1}{2} \bigg)\\
    &\leq C_{L,T,b,\sigma,\sigma^0} \bigg(\int_s^t\big(1+\big\|X_u\big\|_p\big)du + \bigg[ \int_s^t\big(1+\big\|X_u\big\|_p\big)^2du \bigg]^\frac{1}{2} \bigg)\\
    &\leq C_{L,T,b,\sigma,\sigma^0} \bigg(\int_s^t\big(1+\big\|X_u\big\|_p\big)du + \bigg[ \int_s^t\big(1+\big\|X_u\big\|_p^2\big)du \bigg]^\frac{1}{2} \bigg) \\
    &\leq C_{L,T,b,\sigma,\sigma^0}\bigg[ \bigg(1+\bigg\|\underset{0\leq t\leq T}{\sup}|X_t|\bigg\|_p\bigg)(t-s) + \bigg(1+\bigg\|\underset{0\leq t\leq T}{\sup}|X_t|\bigg\|_p^2\bigg)^\frac{1}{2}\sqrt{t-s} \bigg] \\
    &\leq \kappa\, \sqrt{t-s}\quad \text{(by Inequality \eqref{eq : estimate MKV solution})}.
\end{align*}
where $\kappa$ is a positive constant depending on $p$, $L$, $T$, $b$, $\sigma$, $\sigma^0$, and $\|X_0\|_p$. Then we have
\[ \big\|X_t-X_s\big\|_p \leq \kappa \sqrt{t-s}. \hfill \qedhere \] 
\end{proof}

\begin{proof}[Proof of Lemma \ref{lem:MKV sol cond iid}]

The particles $(\bar{Y}^i_{t})_{1\leq i\leq N}$ are copies of the continuous of the continuous extension $\bar{X}$ of the Euler scheme \eqref{eq:continuous euler scheme MKV}. From \eqref{eq:continuous euler solution scheme}, for every $1\leq i\leq N$, there exists a measurable function $\Psi^i$ on $\R^d \times \mathcal{C}([0,T],\R^q) \times \mathcal{C}([0,T],\R^q)$ such that:
\[ \bar{Y}^i = \Psi^i\Big(X_0^i, (W^i_t)_{0\leq t \leq T}, (W^0_t)_{0\leq t \leq T}\Big).\]
As the idiosyncratic noises and the initial random variables $(W^i,X_0^i)_{1\leq i\leq N}$ are i.i.d. by definition, the particles $(\bar{Y}^i)_{1\leq i\leq N}$ are identically distributed and independent conditionally to $W^0$.
\end{proof}

\printbibliography

\end{document}